\documentclass[10pt]{amsart}
\usepackage[pagebackref]{hyperref}

\renewcommand{\d}{\mathrm{d}}

\newcommand{\cA}{{\mathcal A}}

\newcommand{\cI}{{\mathcal I}}
\newcommand{\cV}{{\mathcal V}}

\newcommand{\bbC}{{\mathbb C}}

\newcommand{\bbP}{{\mathbb P}}
\newcommand{\bbR}{{\mathbb R}}

\newcommand{\bbZ}{{\mathbb Z}}

\newcommand{\GL}{\operatorname{GL}}
\newcommand{\Gr}{\operatorname{Gr}}
\newcommand{\SL}{\operatorname{SL}}
\newcommand{\SO}{\operatorname{SO}}
\newcommand{\CO}{\operatorname{CO}}
\newcommand{\SU}{\operatorname{SU}}
\newcommand{\Un}{\operatorname{U}}

\newcommand{\Ric}{\operatorname{Ric}}
\newcommand{\Ad}{\operatorname{Ad}}

\newcommand{\Hess}{\operatorname{Hess}}

\DeclareMathOperator{\tr}{tr}

\DeclareMathOperator{\Hom}{Hom}

\newcommand{\phm}{\phantom{-}}

\newcommand{\eugl}{\operatorname{\mathfrak{gl}}}
\newcommand{\euco}{\operatorname{\mathfrak{co}}}
\newcommand{\euso}{\operatorname{\mathfrak{so}}}
\newcommand{\eusu}{\operatorname{\mathfrak{su}}}

\newcommand{\euh}{\operatorname{\mathfrak{h}}}

\newcommand{\eum}{\operatorname{\mathfrak{m}}}

\newcommand{\w}{{\mathchoice{\,{\scriptstyle\wedge}\,}{{\scriptstyle\wedge}}
      {{\scriptscriptstyle\wedge}}{{\scriptscriptstyle\wedge}}}}

\newcommand{\be}{\begin{equation}}
\newcommand{\ee}{\end{equation}}
\newcommand{\bpm}{\begin{pmatrix}}
\newcommand{\epm}{\end{pmatrix}}

\numberwithin{equation}{section}

\newtheorem{theorem}{Theorem}

\newtheorem{proposition}{Proposition}

\theoremstyle{remark}

\newtheorem{remark}{Remark}
\newtheorem{example}{Example}

\begin{document}

\author[R. Bryant]{Robert L. Bryant}
\address{Duke University Mathematics Department\\
         PO Box 90320\\
         Durham, NC 27708-0320}
\email{\href{mailto:bryant@math.duke.edu}{bryant@math.duke.edu}}
\urladdr{\href{http://www.math.duke.edu/~bryant}%
         {http://www.math.duke.edu/\lower3pt\hbox{\symbol{'176}}bryant}}

\title[Notes on EDS]
      {Notes on Exterior Differential Systems}

\date{December 12, 2013}

\begin{abstract}
These are notes for a very rapid introduction 
to the basics of exterior differential systems 
and their connection with what is now known as Lie theory, 
together with some typical and not-so-typical applications
to illustrate their use.
\end{abstract}



\thanks{
These notes were mostly written during a December 2013 workshop 
\emph{Exterior Differential Systems and Lie Theory} 
at the Fields Institute in Toronto, my participation in which 
was supported by the Clay Mathematics Institute.  
My thanks go to the CMI as well as to the National Science Foundation 
for its support via the grant DMS-1359583.  
\hfill\break
\hspace*{\parindent} 
This is Draft Version~$0.3$, last modified on May 12, 2014.
}

\maketitle

\setcounter{tocdepth}{2}
\tableofcontents

\section{Introduction}\label{sec: intro}

Around the beginning of the 20th century, 
\'Elie Cartan developed a theory of partial differential equations
that was well-suited for the study of local problems in differential
geometry.  His fundamental insight was that many geometric problems
(roughly speaking, those that were independent of choice of coordinates) 
could be recast as problems 
in which one has a set of functions and differential forms 
satisfying some given set of `structure equations', 
i.e., conditions on the exterior derivatives 
of the given functions and forms.  
The invariance properties of the exterior derivative 
then made possible an approach to differential equations 
that Cartan then developed and applied 
in a very large number of situations, beginning with his theory
of `infinite groups' (what we now call pseudo-groups) 
and continuing throughout his later work in classical differential
geometry.

In Cartan's formulation, he was mainly concerned with systems 
that consisted of functions and $1$-forms, and their exterior derivatives.
Erich K\"ahler~\cite{Kahler1934} realized that Cartan's theory 
could be usefully extended to systems generated by forms of arbitrary degree. 
The resulting extension is now known as Cartan-K\"ahler Theory 
and Cartan's general approach is now generally 
called the theory of `exterior differential systems', or simply `EDS'  

Cartan and his students used EDS with great success in the cases 
in which a given problem could be recast as a system satisfying
the hypothesis of \emph{involutivity}, 
which many important systems did.  
Cartan introduced the process of \emph{prolongation}
as an algorithm whose purpose is to replace a given exterior differential 
system with one that has essentially the same solutions 
but is also involutive, the idea being that all of the solutions 
of any system should be describable as the solutions of an involutive system.  
Cartan was never able to prove that his prolongation process 
succeeded in all cases, but, later, Masatake Kuranishi~\cite{Kur} 
did succeed in proving a version of the desired prolongation theorem, 
one that applies in nearly all cases of interest, 
thus essentially completing the theory on an important point.

In these notes, I introduce the basics of exterior
differential systems, covering the essential definitions and theorems,
but I do not attempt to discuss the proofs of fundamental results 
such as Cartan's Bound (aka Cartan's Test), the Cartan-K\"ahler Theorem,
or the Cartan-Kuranishi Theorem.  
For proofs of these, the reader can consult any of the standard sources 
in the subject, such as~\cite{BCGGG}.

My choice of subject matter and examples was heavily influenced
by the audience of the workshop at which this material was written,
and I have made no attempt to give a more comprehensive view 
of the standard topics in exterior differential systems.  
Rather, I have focused on applications to `Cartan structure equations'
associated to differential geometric problems and tried to show
how Cartan's theory is connected with modern-day Lie theory.

\subsection{Differential ideals}\label{ssec: diffideals}
Let $M^n$ be a smooth $n$-manifold.  
An \emph{exterior differential system} on~$M$ 
is a graded, differentially closed ideal~$\cI\subset\cA^*(M)$.  

While it is not strictly necessary, it simplifies some statements
if one assumes that $\cI$ is generated in positive degrees, 
i.e., $\cI^0 = \cI\cap \cA^0(M) = (0)$, 
so I will assume this throughout the notes.

\subsection{Integral manifolds and elements}\label{ssec: intmfds}
An \emph{integral manifold} of~$\cI$ is a submanifold~$f:N\to M$
such that $f^*(\phi) = 0$ for all $\phi\in\cI$.  

\begin{remark}
In most applications of exterior differential systems, 
the integral manifolds of a certain dimension 
(often the maximal dimension) of a given differential ideal~$\cI$ 
represent the local solutions of some geometric problem 
that can be expressed in terms of partial differential equations.
Thus, one is interested in techniques 
for describing the integral manifolds of a given~$\cI$.
\end{remark}

An \emph{integral element} of~$\cI$ is a $p$-plane~$E\subset \Gr_p(TM)$
such that $\iota_E^*(\phi) = 0$ for all $\phi\in\cI$.  
The set of $p$-dimensional integral elements of $\cI$ 
is a closed subset $\cV_p(\cI)\subseteq\Gr_p(TM)$.  
(It is not always a smooth submanifold of this bundle.)

\begin{remark}
Every tangent plane to an integral manifold of~$\cI$ 
is an integral element of~$\cI$.  The fundamental problem 
of exterior differential systems is to decide whether, 
for a given $E\in\cV_p(I)$, there is an integral manifold of~$\cI$ 
that has~$E$ as one of its tangent spaces.
\end{remark}

\subsection{Polar spaces}\label{ssec: polarspaces}
Fix~$E\in\cV_p(\cI)$, with $E\subset T_xM$, 
and let $e_1,\ldots,e_p$ be a basis of~$E$.  
The \emph{polar space} (sometimes called the \emph{enlargement space}) 
of~$E$ is the subspace
$$
H(E) = \{ v\in T_x M\ |\ \phi(v,e_1,\ldots,e_p) = 0\ 
            \forall \phi\in \cI^{p+1}\ \} \subset T_xM.
$$
From its definition, any $E_+\in\cV_{p+1}(\cI)$ that contains~$E$ 
must be contatined in~$H(E)$ and, conversely, any $E_+\in\Gr_{p+1}(TM)$
that satisfies $E\subset E_+\subseteq H(E)$ satisfies~$E_+\in\cV_{p+1}(\cI)$. 
Set $c(E) = \dim\bigl(T_xM/H(E)\bigr)$.

While determining the structure of~$\cV_p(\cI)$ can be difficult,
one sees that the problem of understanding the $(p{+}1)$-dimensional
extensions that are integral elements 
of a given $p$-dimensional integral element is essentially a linear one.

\subsection{Cartan's Bound and characters}\label{ssec: CartanBound}
Let $E\in\cV_n(I)$ be fixed, and let $F = (E_0,E_1,\ldots,E_{n-1})$
be a flag of subspaces of~$E$, with $\dim E_i = i$.  Thus,
$$
(0)_x = E_0 \subset E_1\subset\cdots\subset E_{n-1}\subset E\subset T_xM.
$$
Note that $E_i$ belongs to $\cV_i(I)$.  The following result is
due to Cartan and K\"ahler.

\begin{proposition}[Cartan's Bound]
Given $E\in\cV_n(\cI)$ and a flag~$F = (E_i)$ in~$E$ as above,
there is an open $E$-neighborhood~$U\subset\Gr_n(TM)$ 
such that $\cV_n(\cI)\cap U$ is contained in a smooth submanifold of~$U$
of codimension
$$
c(F) = c(E_0) + c(E_1) + \cdots + c(E_{n-1}).
$$
\end{proposition}  

If $V_n(\cI)$ near~$E$ actually \emph{is} 
a smooth submanifold of~$\Gr_n(TM)$ of codimension~$c(F)$, 
then $E$ is said to be \emph{Cartan-ordinary}  
and the flag~$F$ is said to be a \emph{regular flag} of~$E$.

Let~$\cV^o_n(\cI)\subset\cV_n(\cI)$ denote the subset 
consisting of Cartan-ordinary integral elements of~$\cI$. 
It is an open (but possibly empty) subset of~$\cV_n(\cI)$ 
that is a smooth submanifold of~$\Gr_n(TM)$,
and the basepoint projection~$\pi:\cV^o_n(\cI)\to M$ is a submersion.
(This is because of the standing assumption that~$\cI^0=(0)$.) 

A `dual' version of Cartan's Bound (also known as `Cartan's Test') 
is often useful:
For any~$E\in\cV_n(\cI)$ and any flag~$F=(E_0,E_1,\ldots,E_{n-1})$,
the \emph{character sequence} of~$F$ is the sequence of nonnegative
integers
$$
\bigl(s_0(F),s_1(F),\ldots,s_n(F)\bigr)
$$
such that
$$
s_i(F) = \left\{ 
\begin{aligned} 
&c(E_0) &\quad i=0,\\
&c(E_i)-c(E_{i-1}) &\quad 1\le i < n,\\
&\dim H(E_{n-1}) - n  &\quad i=n.
\end{aligned} \right.
$$ 
Then Cartan's bound can also be expressed as saying that, near~$E$,
the subset~$\cV_n(\cI)$ is contained in a submanifold of~$\Gr_n(TM)$
of dimension
$$
\dim M + s_1(F) + 2s_2(F) + \cdots + ns_n(F),
$$
and that, if, near~$E$, the subset~$\cV_n(\cI)$ \emph{is} 
a submanifold of~$\Gr_n(TM)$ of this dimension, then $E$ 
is Cartan-ordinary and the flag~$F$ is regular.

When $E$ is Cartan-ordinary,
the character sequence $\bigl(s_k(F)\bigr)$ 
is the same for all regular flags $F=(E_0,E_1,\ldots,E_{n-1})$ 
in~$E$.  This common sequence is known as the sequence
of \emph{Cartan characters} of~$E$ 
and simply written as the sequence~$\bigl(s_k(E)\bigr)$. 
Moreover, the characters~$s_k$ are constant 
on the connected components of~$\cV^o_n(\cI)$.

\section{Cartan-K\"ahler Theory}

\subsection{A form of the Cartan-K\"ahler Theorem}
The main result needed in these notes is the following version 
of the Cartan-K\"ahler Theorem.%
\footnote{The standard, slightly stronger version 
of the Cartan-K\"ahler Theorem has more technical hypotheses 
and so takes a bit longer to state.}

\begin{theorem}[Cartan-K\"ahler]
Suppose that $\cI$ is a real-analytic exterior differential system on~$M$
that is generated in positive degree 
and that $E\in\cV_n(\cI)$ is Cartan-ordinary.  
Then there exists a  real-analytic integral manifold of~$\cI$ 
that has $E$ as one of its tangent spaces.
\end{theorem}

\begin{remark}[Generality]
The Cartan-K\"ahler theorem constructs the desired integral manifold
by solving a sequence of initial value problems via the Cauchy-Kowalevski
Theorem.  At each step in the sequence, one gets to choose 
appropriate initial data that determine the resulting integral manifold. 
In fact, looking at the proof of the Cartan-K\"ahler theorem, 
one sees that there is an open $E$-neighborhood~$U\subset\cV_n(\cI)$
such that the initial data that determines a connected integral manifold
of~$\cI$ whose tangent spaces belong to~$U$
consists of $s_0(E)$ constants, $s_1(E)$ functions of $1$ variable, 
$s_2(E)$ functions of $2$ variables, $\ldots$, 
and $s_n(E)$ functions of $n$ variables that are freely specifiable
(i.e. `arbitrary'), subject only to some open conditions.

Thus, one usually says that the Cartan-ordinary integral manifolds
of~$\cI$ (i.e., the ones whose tangent spaces are Cartan-ordinary)
`depend on $s_0$ constants, $s_1$ functions of $1$ variable, 
$s_2$ functions of $2$ variables, $\ldots$, and $s_n$ functions of $n$ variables'.
\end{remark}

\begin{remark}[Significance of the last nonzero character]
One sometimes encounters statements such as ``only the last nonzero
character really matters,'' which the writer usually phrases as
something like ``the solution depends on $s_q$ functions of $q$ variables''
(where $s_q>0$ and $s_k=0$ for all $k>q$), 
thus ignoring all of the $s_i$ for $i<q$.   

The reason for this is that there is nearly always more than one way 
to describe the local solutions of a given geometric problem 
as the Cartan-ordinary integral manifolds 
of some exterior differential system~$\cI$.  
Two such descriptions might well have different character sequences
(some examples will be given below), 
but they always have the same last nonzero character (at the same level~$q$).  

Nevertheless, for any given exterior differential system~$\cI$, 
the full character sequence does have intrinsic meaning.
\end{remark}

\subsection{Involutive tableau}\label{ssec: InvTab}
Let $V$ and $W$ be vector spaces over~$\bbR$ of dimensions~$n$ and~$m$,
respectively, and let~$A\subset W\otimes V^*$
be an $r$-dimensional linear subspace of the linear maps from~$V$ to~$W$.%
\footnote{In the literature, $A$ is often called a \emph{tableau},
which is simply a borrowing from the French of the word used to describe
a subspace of linear maps written out as a matrix whose entries satisfy
some given linear relations.}  
One wants to understand the space of maps~$f:V\to W$ 
with the property that $f'(x)$ lies in~$A$ for all $x\in V$.
Thus, $f$ is being required to satisfy 
a set of homogeneous, constant coefficient, linear, 
first-order partial differential equations, 
a very basic system of PDE.

Set up an exterior differential system as follows:  
Let $M = W\times V\times A$
and let $u:M\to W$, $x:M\to V$, and $p:M\to A$ denote the projections.
Let $\cI$ be the ideal 
generated by the components of the $W$-valued $1$-form
$\theta = \d u - p\,\d x$.  Thus, $\cI$ is generated 
in degree~$1$ by $m = \dim W$ $1$-forms 
and in degree~$2$ by the (at most) $m$ independent $2$-forms 
that are the components of $\d\theta = -\d p \w \d x$.

An $n$-plane $E\in\Gr_n(TM)$ at $(u_0,x_0,p_0)\in M$ 
on which the components of $\d x$ are independent 
will be described by equations of the form
$$
\d u - q(E)\,\d x = \d p - s(E)\,\d x = 0
$$
where $q(E)$ belongs to $W\otimes V^*$ and $s(E)$ 
belongs to $A\otimes V^*\subset (W\otimes V^*)\otimes V^*$.  
It will be an integral element of~$\cI$ if and only if, 
first $q(E)=p_0$, and, second $\bigl(s(E)\,\d x\bigr)\w \d x = 0$.  
This last condition is equivalent 
to requiring that $s(E)$ lie in the intersection 
$$
A^{(1)} = (A\otimes V^*) \cap\bigl(W\otimes S^2(V^*)\bigr)
$$
Let $r^{(1)}$ denote the dimension of this space.  
Thus, the space~$\cV_n(I)$ near~$E$ is a submanifold of~$\Gr_n(TM)$
of codimension $mn + (rn - r^{(1)})$.

Now let $(0)=V_0\subset V_1\subset\cdots\subset V_{n-1}\subset V_n = V$
be a flag in $V$, and, for each $k$, 
let $A_k\subset W\otimes V^*_k$ denote the image of~$A$ 
under the projection~$W\otimes V^*\to W\otimes V^*_k$.  
This defines a flag~$F = (E_0,E_1,\ldots,E_{n-1})$
in any~$E\in \cV_n(I)$ on which $\d x :E \to V$ 
is an isomorphism by letting $\d x(E_i) = V_i$.  
Inspection now shows that $c(E_i) = m + \dim A_i$, 
so Cartan's bound becomes
$$
mn + (rn - r^{(1)}) \ge c(E_0)+\cdots+c(E_{n-1}) 
 = mn + \sum_{i=1}^{n-1}\dim A_i\,,
$$
which, after rearrangement, becomes
$$
\dim A^{(1)} = r^{(1)} \le n\dim A - \sum_{i=1}^{n-1}\dim A_i\,.
$$
In particular, whether an integral element 
on which $\d x$ is independent has a regular flag 
(and hence is Cartan-ordinary) 
depends only on the subspace~$A\subset W\otimes V^*$.

The numbers $s_i(F,A) = \dim A_i - \dim A_{i-1}$ for $1\le i\le n$
are called the \emph{characters} of the flag~$F$ with respect to~$A$.  
In terms of the characters~$s_i(F,A)$, the above inequality becomes
$$
\dim A^{(1)} \le s_1(F,A) + 2\,s_2(F,A) + \cdots + n\,s_n(F,A)
$$
and equality holds if and only if $F$ is a regular flag and
the integral elements~$E\in\cV_n(I)$ on which $\d x:E\to V$ 
is a isomorphism are Cartan-ordinary.  When such a flag~$F$ exists,
the tableau~$A$ is said to be \emph{involutive}, and the
\emph{Cartan characters} of~$A$ are $s_i(A) = s_i(F,A)$
(computed with respect to any regular flag~$F$).

When $A$ is involutive, the Cartan-K\"ahler theorem
implies that the real-analytic integral manifolds of~$A$ exist 
and depend on $s_0=m$ constants, $s_1(A)$ functions of $1$ variable,
$s_2(A)$ functions of $2$ variables, etc.  

In particular, if one takes the Taylor series of the `general' solution
$f:V\to W$ of the equations forcing $f'(x)$ to lie in~$A$ for all $x$,
one gets
$$
f(x) = f_0 + f_1(x) + f_2(x) + \cdots + f_k(x) + \cdots
$$
where $f_k$ is a $W$-valued homogeneous polynomial 
of degree $k$ on~$V$ and hence lies in the subspace
$$
A^{(k-1)} = \bigl(W\otimes S^k(V^*)\bigr)\cap\bigl(A\otimes S^{k-1}(V^*)\bigr).
$$
which has dimension
$$
\dim A^{(k-1)} = \sum_{j=1}^n {{j+k-2}\choose{k-1}}s_j(A)\,,
$$
which is exactly what one would expect if $f$ 
were to be thought of as being comprised of $s_1(A)$ functions of $1$ variable,
$s_2(A)$ functions of $2$ variables, etc.

The concept of involutivity turns out to be fundamental, 
so it is worthwhile to examine how this is connected
with the notion of a Cartan-ordinary integral element in general.

Thus, fix~$E\in\cV_n(\cI)$, with $E\subset T_xM$.  
Let~$E^\perp\subset T^*_xM$ 
be the space of forms that vanish on~$E$, and note that the
ideal~$(E^\perp)\subset\Lambda(T^*_xM)$ generated by~$E^\perp$ 
consists of the forms that vanish on~$E\subset T_xM$. 
Let~$\cI_x\subset \Lambda(T^*_xM)$ 
denote the set of values of forms in~$\cI$ at the point~$x$,
Then $\cI_x$ is contained in~$(E^\perp)$ because $E$ is an integral
element of~$\cI$. Consider the quotient
$$
I_E = (\cI_x+(E^\perp)^2)/(E^\perp)^2 
\subseteq (E^\perp)/(E^\perp)^2\simeq E^\perp\otimes \Lambda(E^*).
$$
This $I_E\subset E^\perp\otimes\Lambda(E^*)$ 
should be thought of as the `linearization' of the ideal~$\cI$ at~$E$.
It generates an ideal~$(I_E)$ in the space of forms on~$T_xM/E\oplus E$ (whose dual space is $E^\perp\oplus E^*$) that has $0\oplus E\subset
T_xM/E\oplus E$ as an integral element.

If $E$ is, in addition, Cartan-ordinary, then it is not difficult
to show that $0\oplus E$ is a Cartan-ordinary integral element of
$(I_E)$, that a flag of the form $0\oplus E_i$ is regular for $0\oplus E$
if and only if $F = (E_i)$ is a regular flag for $E$, and that one
has equality of characters $s_i(0\oplus E) = s_i(E)$.

This motivates the following:  Given two vector spaces~$W$
and $V$ (of dimensions $m$, and $n$, respectively), a graded%
\footnote{This means that $I$ is the direct sum of its 
subspaces $I^q = I\cap \bigl(W^*{\otimes}\Lambda^q(V^*)\bigr)$.}
subspace $I\subset W^*\otimes \Lambda(V^*)$ 
is \emph{involutive} if $0\oplus V\subset W\oplus V$ 
is a Cartan-ordinary integral element 
of the ideal $(I)\subseteq\Lambda\bigl((W\oplus V)^*\bigr)$ 
generated by~$I$.
In this case, set $s_i(I) = s_i(0\oplus V)$ and let
$$
A_I=\left\{ f\in\Hom(V,W)\ |\ \Gamma_f \in \cV_n\bigl((I)\bigr)\ \right\}
\subset W\otimes V^*,
$$
where $\Gamma_f = \{ (f(x),x)\ |\ x\in V\ \} \subset W\oplus V$;
the subspace $A_I$ is said to be the \emph{tableau} of~$I$.  
When $I$ is involutive, $A_I$ is also involutive 
and its characters are
$$
s_i(A_I) = s_i(I) + s_{i+1}(I) + \cdots + s_n(I).
$$

\section{First Examples and Applications}

I will now give a basic set of examples
illustrating the concepts and applications of the Cartan-K\"ahler Theorem.
Some are intended just to help the reader gain familiarity with the concepts, 
while others will turn out to have significant applications.

\begin{example}[The Frobenius theorem]
Suppose that $\cI$ on~$M^{n+s}$ 
can be locally generated \emph{algebraically}
by~$s$ linearly independent $1$-forms~$\theta^1,\ldots,\theta^s$. 
In particular, since~$\cI$ is differentially closed, 
it follows that there are (local) $1$-forms~$\phi^a_b$ 
such that $\d\theta^a = \phi^a_b\w\theta^b$. 

In particular, there is a unique $n$-dimensional integral element
at each point~$x\in M$, 
namely the $n$-dimensional subspace $E_x\subset T_xM$ 
on which each of the $\theta^a$ vanish.  
Thus,~$\cV_n(\cI)\subset\Gr_n(TM)$ is simply a copy of~$M$, 
in fact, the image of a smooth section of the bundle~$\Gr_n(TM)$,
so it is a smooth manifold of dimension~$n{+}s$.  Meanwhile,
for any flag~$F=(E_0,\ldots,E_{n-1})$ in~$E_x$, 
one has $H(E_p) = E_x$, so $c(E_i) = s$ for $0\le i<n$.  
In particular, $s_0(F)=s$ and $s_i(F)=0$ for $0<i<n$. 
Since $\dim\cV_n(\cI) = n{+}s = \dim M + s_1 + 2s_2 + \cdots + ns_n$,
it follows that Cartan's bound is saturated, and all of the elements
of~$\cV_n(\cI)$ are Cartan-ordinary and all their flags are regular.

By the Cartan-K\"ahler Theorem, every $E_x$ is tangent to an integral
manifold of~$\cI$ and the local integral manifolds near~$E$ depend on
$s_0=s$ constants.

Now, in this particular case, there is another way to get the same result,
which is to use the Frobenius Theorem (which is even better 
since it applies in the smooth setting).
This Theorem says that, locally, it is possible 
to choose closed generators~$\theta^a=\d y^a$ 
for some functions~$y^1,\ldots, y^s$ 
that form part of a coordinate system~$x^1,\ldots,x^n,y^1,\ldots,y^s$.  
Then the local $n$-dimensional integral manifolds of~$\cI$ 
are the leaves defined by holding the $y^a$ constant,
so that the `general' local $n$-dimensional integral manifold 
depends on $s$ constants, in agreement with the prediction
of the Cartan-K\"ahler Theorem.
\end{example}

\begin{example}[A non-ordinary integral element]
Let $M=\bbR^3$, with coordinates $x,y,z$, 
and let $\cI$ be generated 
by the $2$-forms $\d x\w \d z$ and $\d y\w \d z$.  
Then the $2$-plane field defined by $\d z = 0$ 
consists of $2$-dimensional integral elements, 
and these are the only $2$-dimensional integral elements, 
so that $\cV_2(\cI)$ is a smooth $3$-manifold in $\Gr_2(TM)$.  
Since $\cI^1 = 0$, one has $c(E_0) = 0$
for all $E_0$.  Letting $E_1$ be spanned by $a\partial_x + b\partial_y$, 
where $(a,b)\not=0$, one finds that $H(E_1)$ has dimension $2$
and is defined by $\d z = 0$, so~$c(E_1) = 1$.  
Thus, $c(E_0)+c(E_1) = 1$ 
while the codimension of $\cV_2(\cI)$ in $\Gr_2(TM)$ is $2$. 
Thus, $E_2 = H(E_1)$ has no regular flag and hence is not Cartan-ordinary.

It may seem disappointing 
that the Cartan-K\"ahler Theorem does not apply 
to prove the existence of $2$-dimensional integral manifolds,
especially, since there evidently does exist an integral manifold
tangent to every $2$-dimensional integral element, namely, 
a horizontal plane $z=z_0$.

However, to see why one should not expect 
Cartan-K\"ahler to apply in this case,
consider a modification of this example 
got by instead considering the ideal~$\cI'$ 
generated by  $\d x\w \d z$ and $\d y\w (\d z - y\,\d x)$. 
The ideals~$\cI$ and $\cI'$ are algebraically equivalent at each point, 
and so one sees that, there is also a unique $2$-dimensional integral
element of~$\cI'$ through each point, namely the one that satisfies 
$\d z - y\,\d x = 0$.  Since the algebra of the polar equations 
is essentially the same for~$\cI'$ as it is for~$\cI$, 
these integral elements of~$\cI'$ also are not Cartan-ordinary, 
and this is good because there evidently 
are not any integral surfaces of the equation~$\d z - y\,\d x=0$.
\end{example}

\begin{example}[Lagrangian submanifolds]
Let $M = \bbR^{2n}$ and let $\cI$ be generated by the symplectic form
$$
\Omega = \d p_1\w\d x^1 +\cdots + \d p_n\w\d x^n.
$$
Then the $n$-plane $E$ spanned by the $\partial_{x^i}$ 
is an integral element of~$\cI$, 
and, if one takes the flag $F = (E_0,\ldots,E_{n-1})$
so that $E_i$ is spanned by the $\partial_{x^j}$ with $1\le j\le i$, 
then one computes that, for $0<i<n$, the polar space~$H(E_i)$ 
is the subspace defined by $\d p_1=\d p_2=\cdots=\d p_i=0$. 
Thus, $c(E_i) = i$.  

By Cartan's bound, $\cV_n(\cI)$ has codimension at least 
$$
C = 1 + 2 + \cdots + (n{-}1) = \tfrac12n(n{-}1) 
$$ 
in $\Gr_n(TM)$ near $E$.  
Meanwhile, any $\tilde E\in\Gr_n(TM)$ on which the $\d x^i$ 
are linearly independent will be defined by unique equations of the form
$$
\d p_i - s_{ij}(\tilde E)\,\d x^i = 0
$$
for some numbers $s_{ij}(\tilde E)$, and these functions $s_{ij}$,
together with the $x^i$ and the~$p_i$ define a local coordinate system
on an open subset of $\Gr_n(TM)$ 
that contains $E$ (which is defined by $s_{ij}(E)=0$).

By Cartan's Lemma, such an $\tilde E$ will be an integral element
of~$\cI$ if and only if $s_{ij}(E)-s_{ji}(\tilde E) = 0$.  
This is $\tfrac12n(n{-}1)$ independent equations on~$\tilde E$, 
so that $\cV_n(\cI)$  has codimension $\tfrac12n(n{-}1) = C$ 
in $\Gr_n(TM)$ near $E$.
Consequently, $E$ is Cartan-ordinary, and the flag $F$ is regular.

Of course, one already knows that Lagrangian manifolds exist, 
so this is not a surprise.  Note, however, 
that what the Cartan-K\"ahler theorem would say 
is that one can specify an integral manifold 
on which the~$x^i$ are independent uniquely by choosing $p_n$ 
to be an arbitrary function of the $x^i$, 
then choosing $p_{n-1}$ subject to the condition that its partial 
in the $x^n$-direction equals the partial
of $p_n$ in the $x^{n-1}$ direction (which determines $p_{n-1}$ 
up to the addition of a function of $x^1,\ldots x^{n-1}$), 
then choosing $p_{n-2}$ subject to the conditions that its partials
in the $x^n$- and $x^{n-1}$-directions are determined by those of $p_n$
and $p_{n-1}$ (which determines $p_{n-2}$ up to the addition 
of a function of $x^1,\ldots x^{n-2}$), etc.  
Thus, the integral manifolds are described 
by a choice of $1= s_n(E)$ function of $n$ variables, 
$1=s_{n-1}(E)$ function of $n{-}1$ variables, etc., 
in agreement with the general theory.

Of course, one can also specify a Lagrangian 
using only one function of $n$ variables 
simply by taking $p_i = \frac{\partial u}{\partial x^i}$
for some function $u$ of $x^1,\ldots,x^n$. 
However, for general~$\cI$, one cannot find such a formula 
that combines the `arbitrary functions' 
in the general Cartan-ordinary integral manifolds of~$\cI$ in this way.

Another way to interpret this `discrepancy' is to note that the Lagrangian
manifolds on which $\d x^1\w\cdots\w\d x^n\not=0$ are, by the above formula
put in correspondence with the arbitrary local function~$u$ of $n$ variables, 
which, via its graph~$\bigl(x^1,\ldots,x^n,u(x^1,\ldots,x^n)\bigr)$ 
is seen to be an integral manifold 
of the trivial ideal~$\cI = (0)$ on~$\bbR^{n+1}$, which has Cartan characters
$$
(s_0,s_1,\ldots,s_{n-1}, s_n) = (0,0,\ldots, 0, 1).
$$
Thus, this provides an example of the phenomenon that I mentioned earlier
of two different exterior differential systems describing (local) solutions
to the same problem.  Note that they have the same last nonzero character,
namely, $s_n=1$, while their lower characters are different.
\end{example}

\subsection{A version of Cartan's Third Theorem}
Suppose that $C^i_{jk}=-C^i_{kj}$ and $F^\alpha_i$ 
(with $1\le i,j,k\le n$ and $1\le \alpha\le s$) 
are given functions on~$\bbR^s$, 
and one wants to know whether or not there exist
linearly independent $1$-forms~$\omega^i$ on~$\bbR^n$ 
and a function $a = (a^\alpha):\bbR^n\to\bbR^s$ 
that satisfy the \emph{Cartan structure equations}
\begin{equation}\label{eq: CSE1}
\d\omega^i = -\tfrac12 C^i_{jk}(a)\,\omega^j\w\omega^k
\qquad\text{and}\qquad
\d a^\alpha = F^\alpha_i(a)\,\omega^i.
\end{equation}
Such a pair~$(a,\omega)$ will be said to be an \emph{augmented coframing}
satisfying the structure equations~\eqref{eq: CSE1}.

Applying the fundamental identity~$\d^2=0$ yields necessary conditions
in order for such a pair~$(a,\omega)$ to exist: 
One must have~$\d(C^i_{jk}(a)\,\omega^j\w\omega^k)=\d(\d\omega^i)=0$ 
and~$\d\bigl(F^\alpha_i(a)\,\omega^i\bigr)=\d(\d a^\alpha)=0$.
Expanding these identities using~\eqref{eq: CSE1} 
and the assumed independence of the~$\omega^i$ then yields that, 
if, for each~$u_0\in\bbR^s$, 
an augmented coframing~$(a,\omega)$ on some $n$-manifold~$M$
exists satisfying~\eqref{eq: CSE1} with $a(x)=u_0$ for some~$x\in M$, 
then one must have
\begin{equation}\label{eq: ddw=0}
 F^\alpha_j{\frac{\partial C^i_{kl}}{\partial u^\alpha}}
+F^\alpha_k{\frac{\partial C^i_{lj}}{\partial u^\alpha}}
+F^\alpha_l{\frac{\partial C^i_{jk}}{\partial u^\alpha}}
=\bigl(C^i_{mj}C^m_{kl}
         +C^i_{mk}C^m_{lj}+C^i_{ml}C^m_{jk}\bigr)
\end{equation}
and
\begin{equation}\label{eq: dda=0}
F^\beta_i{\frac{\partial F^\alpha_j}{\partial u^\beta}}
-F^\beta_j{\frac{\partial F^\alpha_i}{\partial u^\beta}}
=  C^l_{ij}\,F^\alpha_l.
\end{equation}

Cartan proved the converse statement~\cite{Cartan1904}:

\begin{theorem}[Cartan's Third Fundamental Theorem]\label{thm: CTFT}
Suppose that~$C^i_{jk}=-C^i_{kj}$ and~$F^\alpha_i$ 
are real-analytic functions on~$\bbR^s$ 
that satisfy~\eqref{eq: ddw=0} and \eqref{eq: dda=0}.
Then, for any $u_0\in\bbR^s$, 
there exists an augmented coframing~$(a,\omega)$ on~$\bbR^n$ 
that satisfies~\eqref{eq: CSE1} and has $a(0)=u_0$.  
{\upshape(}Moreover, any two such augmented coframings
agree on a neighborhood of~$0\in\bbR^n$ 
up to a diffeomorphism of~$\bbR^n$ that fixes~$0\in\bbR^n$.{\upshape)}
\end{theorem}

\begin{proof}
Let $M = \GL(n,\bbR)\times \bbR^n\times \bbR^s$, and let $p:M\to\GL(n,\bbR)$,
$x:M\to\bbR^n$, and $u:M\to\bbR^s$ be the projections.  
Consider the ideal~$\cI$ generated on~$M$ by the $n$ $2$-forms
$$
\Upsilon^i = \d(p^i_j\,\d x^j) 
                 + \tfrac12 C^i_{jk}(u)(p^j_l\,\d x^l)\w(p^k_m\,\d x^m)
$$
and the $s$ $1$-forms
$$
\theta^\alpha = \d u^\alpha - F^\alpha_i(u)\,(p^i_j\,\d x^j).
$$
Note that one can write
$$
\Upsilon^i = \pi^i_j\w\d x^j
$$
for some $1$-forms $\pi^i_j = \d p^i_j + P^i_{jk}\, \d x^k$ for some
functions~$P^i_{jk}$ on~$M$ and that the forms $\pi^i_j$, $\d x^k$, and 
$\theta^\alpha$ define a coframing on~$M$, i.e., they are linearly independent
everywhere and span the cotangent space everywhere.

Now, the hypothesis that $\d^2 = 0$ be a formal consequence of the
structure equations (i.e., the equations~\eqref{eq: ddw=0} 
and~\eqref{eq: dda=0}) is easily seen to be equivalent to the equations 
$$
\d\Upsilon^i = 
\tfrac12 \frac{\partial C^i_{jk}}{\partial u^\alpha} \,
\theta^\alpha\w(p^j_l\,\d x^l)\w(p^k_m\,\d x^m)
+ C^i_{jk}\,\Upsilon^j\w(p^k_m\,\d x^m)
$$ 
and
$$
\d\theta^\alpha 
= \frac{\partial F^\alpha_{i}}{\partial u^\beta}\,
\theta^\beta\w(p^i_j\,\d x^j)
+ F^\alpha_i\,\Upsilon^i.
$$
Thus, these hypotheses imply that $\cI$ 
is generated \emph{algebraically} by the $\Upsilon^i$ 
and the $\theta^\alpha$.
This makes it easy to choose an integral element and compute the Cartan
characters:

Fix a point~$z\in M$ and let $E\subset T_zM$ be the $n$-dimensional
integral element defined by $\pi^i_j = \theta^\alpha = 0$.  Let $F$
be the flag in $E$ defined so that $E_i$ is also annihilated by
the $\d x^j$ for $j>i$.  Then one finds that $H(E_i)$
is defined by $\theta^\alpha = \pi^j_k = 0$ where $k\le i$, 
and hence that $c(E_i) = s + ni$ for $0\le i\le n{-}1$.  In particular, 
it follows that $\cV_n(\cI)$ must be contained in a submanifold $\Gr_n(TM)$
of codimension at least~$C = ns + \tfrac12n^2(n{-}1)$.  

Meanwhile, any $n$-plane $\tilde E$ on which the $\d x^i$ 
are linearly independent is specified 
by knowing the $ns +n^3$ numbers~$s^\alpha_i(\tilde E)$
and~$s^i_{jk}(\tilde E)$ such that $\tilde E$ satisfies
$$
\pi^i_j - s^i_{jk}(\tilde E)\, \d x^k
 = \theta^\alpha - s^\alpha_k(\tilde E),\d x^k = 0.
$$
The condition that such an $\tilde E$ be an integral element of~$\cI$
is then that $s^\alpha_k(\tilde E) = s^i_{jk}(\tilde E) - s^i_{kj}(\tilde E) 
= 0$, which is $ns + \tfrac12n^2(n{-}1)=C$ equations
on~$\tilde E$.  Thus, $E$ is Cartan-ordinary, and $F$ is a regular flag.

Now, since the functions $C^i_{jk}$ and $F^\alpha_i$ 
are assumed to be real-analytic, the Cartan-K\"ahler Theorem applies 
and one concludes that there 
is an integral manifold of~$\cI$ tangent to~$E$.  This integral
manifold is described by having the $p^i_j$ and the $u^\alpha$ 
be certain functions of the $x^1,\ldots, x^n$, 
say, $p^i_j = f^i_j(x)$ and $u^\alpha = a^\alpha(x)$.  
These then give the desired $(a^\alpha,\omega^i)
=\bigl(a^\alpha(x),f^i_j(x)\,\d x^j\bigr)$.
\end{proof}

\begin{remark}[Cartan's original theorem]
The result I have just proved%
\footnote{The proof in the text is Cartan's; 
I have merely simplified his proof as possible in this special case.}
is only a very special case of the theorem that Cartan proves 
in the first of his `infinite groups' papers~\cite{Cartan1904}.  
However, this version suffices for applications 
that I have in mind in these notes, 
and it is much easier to state than Cartan's full theorem.
\end{remark}

\begin{remark}[Uniqueness]
The general Cartan-ordinary integral of~$\cI$ 
depends on $n$ arbitrary functions of $n$ variables
(since the last nonzero character is $s_n=n$), 
but this is to be expected because,
given any integral as a (local) coframing on~$\bbR^n$, 
one can get others by simply pulling back 
by an arbitrary diffeomorphism of~$\bbR^n$.%
\footnote{Alternatively, one should think of the
Cartan-ordinary integral manifolds of~$\cI$ 
as giving a (local) augmented coframing~$(a,\omega)$ 
satisfying the structure equations~\eqref{eq: CSE1} \emph{plus}
a local coordinate system~$x = (x^i)$ on the domain of~$(a,\omega)$.}
To get uniqueness up to local diffeomorphism for data~$(a,\omega)$ 
in which $a$ takes on a specific value~$a_0\in\bbR^s$, 
one shows that two such solutions are locally equivalent 
by an application of Cartan's technique of the graph.

Note, by the way, that when $s=0$ 
(i.e., there are no functions~$a^\alpha$),
this result becomes Lie's Third Theorem 
giving the existence of a local Lie group for any given Lie algebra.
\end{remark}

\begin{remark}[Smoothness and Globalization]
While this treatment assumes real-analyticity, 
so that the Cartan-K\"ahler Theorem can be applied, 
it is now known that the theorem is true 
in the smooth category as well.  
The proof in the smooth case is not difficult, 
but requires a little more insight 
than this simple application of Cartan-K\"ahler.

The reader will probably also have noticed that nothing is
really used about the domain of the functions~$C^i_{jk}$
and~$F^\alpha_i$ other than that it is a smooth manifold of some
dimension~$s$.  This observation spurred the development of
a `globalized' version of Cartan's Theorem, which becomes
the subject of Lie algebroids, in which $\bbR^s$
is replaced by a smooth manifold~$A$.  For these details 
on these developments, as well as the smooth theory,
the reader should consult treatises devoted to these subjects,
but I will sketch the translation here to aid in comparison
with a somewhat generalized construction associated to a
variant of Cartan's Theorem that I will describe in the next
subsection.

Recall that a \emph{Lie algebroid} is a manifold~$A$ 
endowed with a vector bundle~$Y\to A$ of rank~$n$ 
whose space of sections~$\Gamma(Y)$ carries a Lie algebra structure
$$
\bigl\{,\bigl\}:\Gamma(Y)\times\Gamma(Y)\to\Gamma(Y)
$$ 
together with a bundle map~$\alpha:Y\to TA$ 
that induces a homomorphism of Lie algebras
on the spaces of sections%
\footnote{Here, $\Gamma(TA)$, the set of vector fields on~$A$,
is given its standard Lie algebra structure via the Lie bracket.}
and that satisfies 
the Leibnitz compatibility condition
\begin{equation}\label{eq: Leibnitzcomp}
\bigl\{U,fV\bigl\} = \alpha(U)f\,\,V + f\bigl\{U,V\bigl\}
\end{equation}
for all~$U,V\in\Gamma(Y)$ and $f\in C^\infty(A)$.

A \emph{realization} of~$\bigl(A,Y,\{,\},\alpha\bigr)$
is a triple~$(M,a,\omega)$, where~$M$ 
is an $n$-manifold, $a:M\to A$ is a (smooth) mapping,
and~$\omega:TM\to a^*Y$ is a vector bundle isomorphism,
such that~$\alpha\circ\omega = \d a:TM\to TA$ 
and such that $\omega$ induces an isomorphism of Lie algebras
on the space of sections of~$TM$ and $a^*Y$.

To see the translation from Cartan's language 
to that of Lie algebroids, 
start with the data of functions~$C^i_{jk}=-C^i_{kj}$ 
and $F^\alpha_i$ on~$A = \bbR^s$.
Set~$Y = A\times \bbR^n$ with a basis for sections~$U_i$
and set
$$
\bigl\{U_i,U_j\bigl\} = C^k_{ij}\,U_k
$$
and define $\alpha:Y\to TA = T\bbR^s$ by
$$
\alpha(U_i) = F^\alpha_i\,\frac{\partial\hfill}{\partial u^\alpha}.
$$
Then~\eqref{eq: ddw=0} and~\eqref{eq: dda=0} 
are precisely the equations necessary and sufficient
in order that~\eqref{eq: Leibnitzcomp} hold, 
that $\{,\}$ define a Lie bracket on the space of sections of~$Y$,
and that~$\alpha:Y\to TA$ induce a homomorphism of Lie algebras.

Moreover, an augmented coframing~$(a^\alpha,\omega^i)$ 
on a manifold~$M^n$ satisfies Cartan's structure equations 
if and only if, when one sets
$$
\omega = U_i\,\omega^i\,,
$$
and defines~$a:M\to\bbR^s$ to be~$a = (a^\alpha)$,
the data~$(M,a,\omega)$ is a realization in the above sense.

This approach to globalizing Cartan's theorem has been very 
fruitful, and the reader is encouraged to consult the literature
on Lie algebroids for more on this development.

However, it should be borne in mind that 
Cartan's original formulation in terms of
what I am calling `augmented coframings' 
turns out already to be very well suited 
for applications to differential geometry, 
as I hope to show in the discussion of examples below.  
\end{remark}

\subsection{Variants of Cartan's Third Theorem}

Cartan's Third Theorem is one of a number of existence
results that are all proved more or less the same way,
at least in the real-analytic category.  In this subsection,
I will give two such variants, 
and, in the following sections in the notes, 
I will illustrate their use 
in a range of differential geometry problems.

Throughout this first variant, the index ranges $1\le i,j,k\le n$, 
$1\le \alpha\le s$, and $1\le\rho,\sigma\le r$ will be assumed.

Suppose that $C^i_{jk}(u)=-C^i_{kj}(u)$ are given functions on~$\bbR^s$ 
while $F^\alpha_i(u,v)$ are given functions on~$\bbR^{s+r}$, 
and suppose that one wants to know whether or not there exist
linearly independent $1$-forms~$\omega^i$ on an $n$-manifold~$M$,
a function $a = (a^\alpha):M\to\bbR^s$,
and a function~$b = (b^\rho):M\to\bbR^r$ 
that satisfy these \emph{Cartan structure equations}
\begin{equation}\label{eq: CSE2}
\d\omega^i = -\tfrac12 C^i_{jk}(a)\,\omega^j\w\omega^k
\qquad\text{and}\qquad
\d a^\alpha = F^\alpha_i(a,b)\,\omega^i.
\end{equation}
Such a triple~$(a,b,\omega)$ on~$M$ 
will be said to be an \emph{augmented coframing} 
satisfying~\eqref{eq: CSE2}.  

Note that this is a diffeomorphism invariant notion, 
since, if~$f:N\to M$ is a diffeomorphism, 
then~$\bigl(f^*a,f^*b,f^*\omega\bigr)$ 
will be an augmented coframing on~$N$ that satisfies~\eqref{eq: CSE2}.   
In many geometric problems (see some examples in the next section),
one is interested in understanding the `general' augmented
coframing satisfying~\eqref{eq: CSE2} 
and one regards two such augmented coframings 
that differ by a diffeomorphism as equivalent.

One should think of the $b^\rho$ 
as `unconstrained' derivatives of the functions $a^\alpha$.  
Thus, this version of  Cartan's structure equations 
covers situations in the more typical case 
in which one does not have formulae for all of the derivatives 
of the geometric quantities that appear in the problem. 
Informally, one speaks of the functions~$b^\rho$ as `free derivatives'.

To understand necessary and sufficient conditions for such
augmented coframings to exist, one again wants to consider 
the consequences of the identity $\d^2 = 0$, but now, 
because of the free derivatives appearing in the structure equations,
one cannot simply expand this fundamental identity formally 
and arrive at complete necessary and sufficient conditions 
on the functions~$C$ and~$F$.  

Now, the equations $\d(\d\omega^i)=\d(C^i_{jk}(a)\,\omega^j\w\omega^k)=0$ 
do make good sense, 
so one should require that~$C$ and~$F$ at least satisfy
\begin{equation}\label{eq: ddw=0.2}
 F^\alpha_j{\frac{\partial C^i_{kl}}{\partial u^\alpha}}
+F^\alpha_k{\frac{\partial C^i_{lj}}{\partial u^\alpha}}
+F^\alpha_l{\frac{\partial C^i_{jk}}{\partial u^\alpha}}
=\bigl(C^i_{mj}C^m_{kl}
         +C^i_{mk}C^m_{lj}+C^i_{ml}C^m_{jk}\bigr).
\end{equation}
(Because the $F^\alpha_i$ contain the variables~$v^\rho$
while the right hand side of~\eqref{eq: ddw=0.2} does not, 
this equation places constraints on how the~$v^\rho$ 
can appear in the~$F^\alpha_j$.)

Meanwhile, expanding
$\d(\d a^\alpha) = \d\bigl(F^\alpha_i(a,b)\,\omega^i\bigr) = 0$
yields  
$$
0 = {\frac{\partial F^\alpha_i}{\partial v^\rho}}(a,b)\d b^\rho\w\omega^i
+\frac12\left(F^\beta_i(a,b)
{\frac{\partial F^\alpha_j}{\partial u^\beta}}(a,b)
-F^\beta_j(a,b){\frac{\partial F^\alpha_i}{\partial u^\beta}}(a,b)
- C^l_{ij}(a)\,F^\alpha_l(a,b)\right)\,\omega^i\w\omega^j\,,
$$
and the simplest way for these equations to be satisfiable
by some expression of the~$\d b^\rho$ in terms of the~$\omega^i$ 
would be for there to exist functions~$G^\rho_j$ on $\bbR^{s+r}$ 
such that
\begin{equation}\label{eq: Grelns}
F^\beta_i{\frac{\partial F^\alpha_j}{\partial u^\beta}}
-F^\beta_j{\frac{\partial F^\alpha_i}{\partial u^\beta}}
- C^l_{ij}\,F^\alpha_l
=  {\frac{\partial F^\alpha_i}{\partial v^\rho}}G^\rho_j
  -{\frac{\partial F^\alpha_j}{\partial v^\rho}}G^\rho_i\,,
\end{equation}
for then the above equations can be written in the form
$$
0 = {\frac{\partial F^\alpha_i}{\partial v^\rho}}(a,b)
\left(\d b^\rho - G^\rho_j(a,b)\omega^j\right)\w\omega^i.
$$

The conditions~\eqref{eq: ddw=0.2} and~\eqref{eq: Grelns} 
will at least ensure that there are no obvious
incompatibilities derivable by taking the exterior derivatives
of the structure equations.  However, they aren't enough
to guarantee that there won't be higher order incompatibilities.
To rule this out, it will be necessary to impose conditions on
how the `free derivatives' $v^\rho$ appear in the functions~$F^\alpha_i$.
Let~$u_1,\ldots,u_s$ be a basis of $\bbR^s$, and let $x^1,\ldots x^n$
be a basis of the dual of~$\bbR^n$.  
Let~$A(u,v)\subset\Hom(\bbR^n,\bbR^s)$
denote the subspace (i.e., tableau) spanned by the $r$ elements
\begin{equation}\label{eq: Agenerators}
{\frac{\partial F^\alpha_i}{\partial v^\rho}}(u,v)\,u_\alpha\otimes v^i,
\qquad 1\le \rho\le r.
\end{equation}
This~$A(u,v)$ is known as the `tableau of free derivatives' 
of the structure equations at the point~$(u,v)\in\bbR^{s+r}$.

Here is a useful variant%
\footnote{Note that the proof is very closely patterned 
on Cartan's proof of~Theorem~\ref{thm: CTFT}.} 
of Theorem~\ref{thm: CTFT}.

\begin{theorem}\label{thm: CartanThmVar}
Suppose that real analytic functions $C^i_{jk}=-C^i_{kj}$ on~$\bbR^s$
and $F^\alpha_i$ on~$\bbR^{s+r}$ are given satisfying~\eqref{eq: ddw=0.2}
and that there exist real analytic functions~$G^\rho_i$ on~$\bbR^{s+r}$
that satisfy~\eqref{eq: Grelns}.  
Finally, suppose that the tableaux $A(u,v)$ 
defined by~\eqref{eq: Agenerators} 
have dimension~$r$ and are involutive, with Cartan characters $s_i$ 
{\upshape(}$1\le i\le n${\upshape)} for all~$(u,v)\in\bbR^{s+r}$.  
Then, for any $(u_0,v_0)\in\bbR^{s+r}$
there exists an augmented coframing~$(a,b,\omega)$ 
on an open neighborhood~$V$ of~$0$ in~$\bbR^n$
that satisfies~\eqref{eq: CSE2}
and has $\bigl(a(0),b(0)\bigr) = (u_0,v_0)$.
\end{theorem}

\begin{proof}
Let $M = \GL(n,\bbR)\times \bbR^n\times \bbR^s\times\bbR^r$, 
and let $p:M\to\GL(n,\bbR)$,
$x:M\to\bbR^n$, $u:M\to\bbR^s$, and $v:M\to\bbR^r$ be the projections.  
Consider the ideal~$\cI$ generated on~$M$ by the $n$ $2$-forms
$$
\Upsilon^i = \d(p^i_j\,\d x^j) 
                 + \tfrac12 C^i_{jk}(u)(p^j_l\,\d x^l)\w(p^k_m\,\d x^m)
$$
and the $s$ $1$-forms
$$
\theta^\alpha = \d u^\alpha - F^\alpha_i(u,v)\,(p^i_j\,\d x^j).
$$
Note that, as in the proof of Theorem~\ref{thm: CTFT}, one can write
$$
\Upsilon^i = \pi^i_j\w\d x^j
$$
for some $1$-forms $\pi^i_j = \d p^i_j + P^i_{jk}\, \d x^k$ for some
functions~$P^i_{jk}$ on~$M$ and that the forms $\pi^i_j$, $\d x^k$, 
$\theta^\alpha$, together with $\beta^\rho=\d b^\rho- G^\rho_i(p^i_j\,\d x^j)$ define a coframing on~$M$, 
i.e., they are linearly independent everywhere 
and span the cotangent space everywhere.

Now, the hypotheses of the theorem imply that 
$$
\d\Upsilon^i = 
\tfrac12 \frac{\partial C^i_{jk}(u)}{\partial a^\alpha} \,
\theta^\alpha\w(p^j_l\,\d x^l)\w(p^k_m\,\d x^m)
+ C^i_{jk}(u)\,\Upsilon^j\w(p^k_m\,\d x^m)
$$ 
while
$$
\d\theta^\alpha 
= {\frac{\partial F^\alpha_i}{\partial b^\rho}}
\beta^\rho\w(p^i_j\,\d x^j)
+\frac{\partial F^\alpha_{i}(u)}{\partial a^\beta}\,
\theta^\beta\w(p^i_j\,\d x^j)
+ F^\alpha_i\,\Upsilon^i.
$$
Thus, $\cI$ is generated \emph{algebraically} 
by the $\Upsilon^i$, the $\theta^\alpha$, and the $2$-forms
$$
\Theta^\alpha = {\frac{\partial F^\alpha_i}{\partial b^\rho}}
\beta^\rho\w(p^i_j\,\d x^j).
$$
This makes it easy to choose an integral element and compute the Cartan
characters:

Fix a point~$z=(I_n,0,u_0,v_0)\in M$, 
and let $E\subset T_zM$ be the $n$-dimensional
integral element defined by $\pi^i_j = \theta^\alpha = \beta^\rho=0$. 

Choose a regular flag for the tableau~$A(u_0,v_0)$ 
(which, by hypothesis, exists).  
By rotating the $x^i$ if necessary, one can assume
that the flag~$F$ in $E$ defined so that $E_i$ 
is also annihilated by the $\d x^j$ for $j>i$ 
is such a regular flag.  
Then one finds that $H(E_i)$ is defined by 
$$
\theta^\alpha = \pi^j_k 
= {\frac{\partial F^\alpha_k}{\partial b^\rho}}\beta^\rho = 0
$$ 
where $k\le i$, so $c(E_i)
= s+ ni + \dim A(u_0,v_0)_i =  s + ni + s_1+\cdots +s_i$ 
for $0\le i\le n{-}1$.  

Meanwhile, any $n$-plane~$E'\in\Gr_n(TM)$ on which the $\d x^i$
are independent will be defined by equations of the form
$$
\theta^\alpha - q^\alpha_i(E')\,\d x^i
= \pi^l_k - q^l_{ki}(E')\,\d x^i
= \beta^\rho - q^\rho_i(E')\,(p^i_j\,\d x^j) = 0
$$
for some numbers~$q^\alpha_i(E'),q^l_{ki}(E'),q^\rho_i(E')$.
The conditions that~$E'$ be an integral element of~$\cI$
then imply that
$$
q^\alpha_i(E') = q^l_{ki}(E')-q^l_{ik}(E') 
=  {\frac{\partial F^\alpha_i}{\partial b^\rho}}(u,v) q^\rho_j(E')
 - {\frac{\partial F^\alpha_j}{\partial b^\rho}}(u,v) q^\rho_i(E') =0,
$$
and, by the hypothesis that~$F = (E_i)$ is a regular flag 
for the tableaux~$A(u_0,v_0)$ (and hence is also regular for~$A(u,v)$ 
for $(u,v)$ near~$(u_0,v_0)$), it follows that this is
$$
c(E_0) + c(E_1) +\cdots + c(E_{n-1}) 
= ns+\tfrac12n^2(n{+}1) + (n{-}1)s_1 + (n{-}2)s_2 + \cdots + s_{n-1}
$$
equations on the quantities $q^\alpha_i(E'),q^l_{ki}(E'),q^\rho_i(E')$.
Thus, Cartan's bound is saturated, and the flag~$F$ is regular for~$E$.%
\footnote{Essentially, the involutive tableau~$A(u_0,v_0)$ 
is being combined with a tableau 
already shown to be involutive in the proof of Theorem~\ref{thm: CTFT},
one for which \emph{every} flag is regular.  Perhaps, I should also 
remind the reader that the $s_i$ are the characters of the tableaux~$A(u,v)$
and \emph{not} of the ideal~$\cI$ constructed above.  In fact, 
one has $s_0(\cI) = s$ and $s_i(\cI) = s_i + n$ for~$1\le i\le n$.}

The ideal~$\cI$ is real analytic, 
so the Cartan-K\"ahler Theorem applies, and one concludes that there 
is an integral manifold of~$\cI$ tangent to~$E$.  This integral
manifold is described by having the $p^i_j$, the $u^\alpha$,
and the $v^\rho$ be functions of the $x^1,\ldots, x^n$, 
say, $p^i_j = f^i_j(x)$ and $u^\alpha = a^\alpha(x)$ and $v^\rho=b^\rho(x)$.  
These then give the desired
augmented coframing~$(a^\alpha,b^\rho,\omega^i)
=\bigl(a^\alpha(x),\ b^\rho(x),\ f^i_j(x)\,\d x^j\bigr)$.
\end{proof}

\begin{remark}[Generality]
Theorem~\ref{thm: CartanThmVar} 
as stated only gives existence for specified $(u_0,v_0)$, 
but, as will be seen, the (local) augmented coframings 
that satisfy the structure equations
depend (modulo diffeomorphism) on $s$ constants, 
$s_1$ functions of $1$ variable, $s_2$ functions of $2$ variables, etc.,
but to make precise sense of this, 
I will need to discuss \emph{prolongation},
which comes in the next section.
\end{remark}

\begin{remark}[Globalization]\label{rem: globalCTV}
Just as in the case of Theorem~\ref{thm: CTFT}, 
which has a modern formulation in terms of Lie algebroids, 
there is a `global' version of Theorem~\ref{thm: CartanThmVar}.%
\footnote{I will not actually need this formulation in these notes,
but since there were questions about this during the lectures, 
I will briefly describe it here.}

The appropriate global data structure, 
$\bigl(A,B,\pi,Y,\{,\},\beta\bigr)$, 
starts with two manifolds,~$A$ of dimension~$s$
and~$B$ of dimension~$r+s$, and a submersion~$\pi:B\to A$.
For notational convenience, let~$K = \ker\pi'\subset TB$,
and let~$Q = TB/K$ be the quotient bundle over~$B$.
For a vector field~$X$ on~$B$, let~$X_K$ (i.e., `$X$ modulo~$K$') 
denote the corresponding section of~$Q$. 

Next, the data structure includes a vector bundle~$Y\to A$ of rank~$n$,
and a Lie algebra structure~$\{,\}$ 
on the space~$C^\infty(Y)$ of sections of~$Y$ over~$A$.
For~$U\in C^\infty(Y)$, let~$U^\pi\in C^\infty(\pi^*Y)$
denote the pullback section of the pullback bundle over~$B$,
i.e., $U^\pi(b) = U\bigl(\pi(b)\bigr)$.

Finally, the data includes a bundle map~$\beta:\pi^*Y\to TB$,
that satisfies
\begin{equation}\label{eq: weakLie}
\beta\bigl(\{U,\,V\}^\pi\bigr)_K
= \bigl[\beta(U^\pi),\beta(V^\pi)\bigr]_K
\end{equation}
and the requirement that there exist an anti-symmetric, bilinear product%
\footnote{N.B.: It is easy to see that there is at most one
such product~$\{\!\{,\}\!\}$ satisfying~\eqref{eq: weakcompat}. 
In general, this `extended' product 
is \emph{not} a Lie algebra structure on~$C^\infty(\pi^*Y)$.}
$\{\!\{,\}\!\}$ on $C^\infty(\pi^*Y)$ that satisfies
the compatibility condition
\begin{equation}\label{eq: weakcompat}
\{\!\{U^\pi,\,f\,V^\pi\}\!\} 
= \bigl(\beta(U^\pi)f\bigr)\ V^\pi + f\,\{U, V\}^\pi
\end{equation}
for~$U,V\in C^\infty(Y)$ and~$f\in C^\infty(B)$.

A \emph{realization} 
of the data structure~$\bigl(A,B,\pi,Y,\{,\},\beta\bigr)$
is a triple~$(M,b,\omega)$, where~$M$ is an $n$-manifold, 
$b:M\to B$ is a smooth mapping, 
and $\omega:TM\to (\pi{\circ}b)^*Y$ is an isomorphism of bundles 
that induces an isomorphism of Lie algebras 
on the appropriate spaces of sections
and that satisfies~$\d (\pi{\circ}b)=\pi'\circ\beta\circ\omega$.

(Note that if~$\pi:B\to A$ is a diffeomorphim (e.g., $r=0$),
then the data $\bigl(A,Y,\{,\},\beta\bigr)$ 
defines a Lie algebroid, 
and the notion of a realization is the standard one.)

Now, there is a map~$\tau:K\to Q\otimes (\pi^*Y)^*$ of~$B$-bundles, 
uniquely determined by the condition that it satisfy
$$
\tau(X)(U^\pi) = [X,\beta(U^\pi)]_K
$$
for any~$X\in\Gamma(K)$ and $U\in C^\infty(Y)$.  
One says that the data~$\bigl(A,B,\pi,Y,\{,\},\beta\bigr)$ 
is \emph{nondegenerate} if~$\tau$ injective, and, further,  
that it is \emph{{\upshape{(}}uniformly{\upshape{)}} involutive}
if $\tau(K)_b\subset Q_b\otimes (\pi^*Y)^*_b$ 
is an involutive tableau for all~$b\in B$ 
(and the Cartan characters~$s_i\bigl(\tau(K)_b)$ 
are constant, independent of~$b\in B$).

Then Theorem~\ref{thm: CartanThmVar} asserts the local existence
of realizations~$(M,b,\omega)$ of uniformly involutive, nondegenerate 
real analytic data structures~$\bigl(A,B,\pi,Y,\{,\},\beta\bigr)$
with~$b:M\to B$ taking any specified value~$b_0\in B$.

To see the translation 
from the notation of Theorem~\ref{thm: CartanThmVar} to this
`global' formulation, let~$A = \bbR^s$ (with coordinates~$u^\alpha)$, 
let $B =\bbR^{s+r}$ (with coordinates $u^\alpha$ and $v^\rho$),
let~$\pi:\bbR^{s+r}\to\bbR^s$ be the projection on the first $s$
coordinates, let~$Y = \bbR^s\times\bbR^n$ (with the standard
basis of sections~$U_i$), let
$$
\{U_i,U_j\} = C^k_{ij}(u)\,U_k\,,
$$
and let
$$
\beta(U^\pi_i) = F^\alpha_i(u,v)\frac{\partial\hfill}{\partial u^\alpha}
                +G^\rho_i(u,v)\frac{\partial\hfill}{\partial v^\rho}.
$$
The reader can now verify that~\eqref{eq: ddw=0.2} and~\eqref{eq: Grelns}
are the necessary and sufficient conditions 
that $\{,\}$ define a Lie bracket on~$C^\infty(Y)$, 
that~\eqref{eq: weakLie} hold, 
and that there exists an extension $\{\!\{,\}\!\}$ 
of $\{,\}$ to sections of~$C^\infty(\pi^*Y)$ 
that satisfies~\eqref{eq: weakcompat}.  

(I should point out that this `global' formulation is not perfect, 
because, ideally, 
one should only have to specify the functions~$G^\rho_i$ 
up to a section of the prolongation of the tableau bundle, 
i.e., one should regard 
two such structures~$\bigl(A,B,\pi,Y,\{,\},\beta\bigr)$ 
and $\bigl(A,B,\pi,Y,\{,\},\tilde\beta\bigr)$ 
as the same if the difference~$\delta\beta=\tilde\beta{-}\beta$,
which is a section of~$TB\otimes(\pi^*Y)^*$,
is actually a section of the kernel~$K^{(1)}$ of the composition
$$
K\otimes (\pi^*Y)^* 
\buildrel\tau\otimes\mathrm{id}\over\longrightarrow
Q\otimes (\pi^*Y)^*\otimes (\pi^*Y)^*
\to Q\otimes\Lambda^2\bigl((\pi^*Y)^*\bigr).
$$
Thus, one should probably formulate the data structure 
with the notion of nondegeneracy built into the axioms
and with~$\beta$ taking values 
in the quotient bundle~$(TB\otimes(\pi^*Y)^*)/K^{(1)}$
instead of in~$TB\otimes(\pi^*Y)^*$.
However, this is turns out to be awkward,
as checking that the axioms even make sense becomes cumbersome.)

While this `global' formulation may be more satisfying
than the `coordinate' formulation in Theorem~\ref{thm: CartanThmVar}, 
one should bear in mind that there is little (and, more often
than not, no) hope of proving the global realization theorems 
that one has in the more familiar case of Lie algebroids.  
For the general such structure, 
there is no obvious notion of completeness of a realization 
and there is also no obvious way to `classify' 
even the germs of realizations up to diffeomorphism.  
(However, there is a way to test two such germs for diffeomorphism
equivalence, at least in the real-analytic category.  
I will say more about this in Remark~\ref{rem: diffequivCTV}.)
\end{remark}

\begin{remark}[Local equivalence of realizations]
\label{rem: diffequivCTV}
The reader may be wondering how one distinguishes two realizations
of the data in Theorem~\ref{thm: CartanThmVar} up to diffeomorphism.
After all, as Cartan proved, given 
two augmented coframings~$(M,a,\omega)$ and $(\bar M,\bar a,\bar\omega)$
satisfying~\eqref{eq: CSE1} and points~$x\in M$ and~$\bar x\in \bar M$
such that~$a(x) = \bar a(\bar x)$, 
there will exist an $x$-neighborhood~$U\subset M$, 
an $\bar x$-neighborhood~$\bar U\subset \bar M$,
and a diffeomorphism~$f:\bar U\to U$
such that~$(\bar a,\bar\omega)=(f^*a, f^*\omega)$
and~$f(\bar x) = x$.

In contrast, for augmented coframings~$(M,a,b,\omega)$ 
and~$(\bar M, \bar a,\bar b,\bar\omega)$ satisfying~\eqref{eq: CSE2},
having points~$x\in M$ and $\bar x\in\bar M$ 
with~$\bigl(a(x),b(x)\bigr) = \bigl(\bar a(\bar x), \bar b(\bar x)\bigr)$
is not sufficient to imply that there is a diffeomorphism~$f:\bar U\to U$ 
for some $x$-neighborhood~$U$ and $\bar x$-neighborhood~$\bar U$
such that~$(\bar a,\bar b,\bar\omega) = (f^*a, f^*b, f^*\omega)$.

A sufficient condition 
(due, of course, to Cartan~\cite{Cartan1937})
for local diffeomorphism equivalence
does exist in this more general case
but it is more subtle.

An augmented coframing~$(a,b,\omega)$ on~$M^n$ 
satisfying~\eqref{eq: CSE2},
is \emph{regular of rank $p$ at~$x\in M$} 
if there is an $x$-neighborhood~$U\subset M$, 
a smooth submersion~$h:U\to\bbR^p$, 
and a smooth map~$(A,B):h(U)\to\bbR^{s+r}$ 
such that~$A:h(U)\to\bbR^s$ is a smooth embedding
and such that~$(a,b) = (A{\circ}h,B{\circ}h)$ holds on~$U$.
Note, in particular, 
that this implies that the image~$(a,b)(U)\subset\bbR^{s+r}$ 
is a smoothly embedded $p$-dimensional submanifold 
that is a graph over its projection~$a(U)\subset\bbR^s$ 
(also a smoothly embedded $p$-dimensional submanifold).
Equivalently,~$(a,b,\omega)$ is regular of rank~$p$ at~$x\in M$ 
if some $p$ of the functions~$a^\alpha$ 
have independent differentials at~$x$ 
and, moreover, on some $x$-neighborhood $U\subset M$, 
all of the other~$a^\alpha$ \emph{and} all of the~$b^\rho$ 
can be expressed as smooth functions of those $p$ independent functions.
For an augmented coframing~$(a,b,\omega)$ satisfying~\eqref{eq: CSE2}, being regular of rank~$p$ at a point~$x\in M$ 
is a diffeomorphism-invariant condition.
  
Cartan showed that, \emph{if} $(M,a,b,\omega)$ 
and~$(\bar M,\bar a,\bar b,\bar\omega)$ satisfy~\eqref{eq: CSE2}, 
are regular of rank~$p$ at points $x\in M$ and $\bar x\in \bar M$ 
with $\bigl(a(x),b(x)\bigr) = \bigl(\bar a(\bar x),\bar b(\bar x)\bigr)$,
and there are an $x$-neighborhood~$U\subset M$ 
and $\bar x$-neighborhood~$\bar U\subset \bar M$
such that~$(a,b)(U)$ and $(\bar a,\bar b)(\bar U)$ 
are the same $p$-dimensional submanifold of~$\bbR^{r+s}$, 
\emph{then}, after possibly shrinking~$U$ and~$\bar U$, 
there exists a diffeomorphism~$f:\bar U\to U$ 
such that $(\bar a,\bar b,\bar\omega)=(f^*a, f^*b, f^*\omega)$
and $f(\bar x) = x$.

The reader should have no trouble rephrasing Cartan's sufficient
condition in a form suitable for the `global data structure' version
described in Remark~\ref{rem: globalCTV}.  (The reader may feel that
the hypotheses of Cartan's equivalence theorem are absurdly strong,
but, without knowing more 
about a specific set of structure equations~\eqref{eq: CSE2}, 
it is not possible to weaken these hypotheses in any significant way
and still get the conclusion of local equivalence, 
as examples show.)
\end{remark}

I conclude this subsection with another useful variant of Cartan's
Third Theorem.  Let $V$ be a vector space of dimension~$n$.  
For each $V$-valued coframing~$\omega:TM\to V$ 
on an $n$-manifold~$M$, 
there will be a unique function~$C:M\to V\otimes\Lambda^2(V^*)$,
the \emph{structure function} of~$\omega$, such that
\begin{equation}\label{eq: domegastreqs}
\d\omega = -\tfrac12C(\omega\w\omega).
\end{equation}
Given a basis $v_i$ of~$V$ with dual basis~$v^i$, 
one has~$\omega=v_i\omega^i$ 
and~$C=\tfrac12C^i_{jk}v_i\otimes{v^j{\w}v^k}$,
and \eqref{eq: domegastreqs} takes 
the familiar form~$\d\omega^i = -\tfrac12 C^i_{jk}\,\omega^j\w\omega^k$.

Now, let~$A\subset V\otimes\Lambda^2(V^*)$ be a submanifold.%
\footnote{In most applications, $A$ will be an affine subspace 
of~$V\otimes\Lambda^2(V^*)$, but the extra generality 
of allowing $A$ to be a submanifold is frequently useful.} 
A $V$-valued coframing~$\omega:TM\to V$ 
will be said to be of \emph{type~$A$} 
if its structure function~$C:M\to V\otimes\Lambda^2(V^*)$ 
takes values in~$A$.  The goal is to determine the generality 
of the space of (local) $V$-valued coframings~$\omega$ 
of type~$A$ when two such that differ by a diffeomorphism of~$M$
are regarded as equivalent.

For example, if~$A$ 
consists of a single point~$a_0=\tfrac12c^i_{jk}\,v_i\otimes{v^j{\w}v^k}$,
then Lie's Theorem asserts that a necessary and sufficient condition
that such a coframing exist is that $J(a_0) = 0$, 
where~$J:V\otimes\Lambda^2(V^*)\to V\otimes \Lambda^3(V^*)$ 
is the quadratic mapping (sometimes called the \emph{Jacobi mapping}) 
that one gets by squaring, contracting, and skewsymmetrizing:
$$
V{\otimes}\Lambda^2(V^*)\to 
\bigl(V{\otimes}\Lambda^2(V^*)\bigr)\otimes 
\bigl(V{\otimes}\Lambda^2(V^*)\bigr)\to
V{\otimes}V^*{\otimes}\Lambda^2(V^*)\to
V{\otimes}\Lambda^3(V^*).
$$
Given a basis~$v_i$ of~$V$ with dual basis~$v^i$, the formula for $J$ is
$$
J\bigl(\tfrac12c^i_{jk}\,v_i\otimes{v^j{\w}v^k}\bigr) = 
\tfrac16(c^i_{jm}c^m_{kl}+c^i_{km}c^m_{lj}+c^i_{lm}c^m_{jk})\,
v_i\otimes{v^j{\w}v^k{\w}v^l}.
$$
Of course, in this case, \emph{all} $V$-valued coframings of type~$A=\{a_0\}$
are locally equivalent up to diffeomorphism.

This motivates the following definitions:  
A submanifold~$A\subset V\otimes \Lambda^2(V^*)$ 
is said to be a \emph{Jacobi manifold} if
\begin{equation}\label{eq: Jacobicond}
J(a)\in \sigma(T_aA\otimes V^*)
\end{equation}
for all~$a\in A$, where 
$\sigma:V\otimes\Lambda^2(V^*)\otimes V^*\to V\otimes\Lambda^3(V^*)$ 
is the skewsymmetrization mapping defined by exterior multiplication.
The condition~\eqref{eq: Jacobicond} is an obvious necessary condition 
in order for there to exist a $V$-valued coframing~$\omega:TM\to V$ 
whose structure function takes values in~$A$ 
and assumes the value~$a\in A$.  It is not, in general, sufficient.

A Jacobi manifold $A$ is \emph{involutive} 
if each of its tangent spaces~$T_a\subset V\otimes\Lambda^2(V^*)$ 
is involutive, with characters~$s_i(T_a)=s_i$.

I can now state a useful existence result that I will apply 
in some examples.

\begin{theorem}\label{thm: CartanThmVar2}
Let $V$ be a vector space, 
and let~$A\subset V{\otimes}\Lambda^2(V^*)$ 
be a real-analytic, involutive Jacobi manifold.  
Then, for any~$a_0\in A$, 
there exists a $V$-valued coframing~$\omega$ of type~$A$
on a neighborhood~$U$ of~$0\in V$ 
such that its structure function~$C$ satisfies $C(0) = a_0$.
\end{theorem}

\begin{proof}
The proof follows the by-now familiar pattern laid
down by Cartan.

The result is local, so one can suppose that~$A$ has dimension~$s$
and is parametrized by a real-analytic embedding~$T:\bbR^s\to A\subset V\otimes\Lambda^2(V^*)$ with~$T(0)=a_0$.  Write
$$
T = \tfrac12 T^i_{jk}(a)\,v_i\otimes v^j\w v^k
$$
where the~$T^i_{jk}$ are some real analytic functions on~$\bbR^s$.
By hypothesis, for each~$a = (a^\alpha)\in\bbR^s$, 
the tableau~$A_a\subset V\otimes\Lambda^2(V^*)$ 
spanned by the $s$ independent elements
$$
A_\alpha(a) = \frac{\partial T^i_{jk}}{\partial a^\alpha}(a)
                    \,v_i\otimes v^j\w v^k\,,
\qquad 1\le\alpha\le s,
$$
is involutive, with characters~$s_i$ for $1\le i\le n$. 
By changing the chosen basis of~$V$ if necessary, 
it can even be supposed that the flag such that~$V_i\subset V$
is spanned by~$v_1,\ldots,v_i$ is a regular flag for~$A_0$ 
(and hence it will be regular for~$A_a$ 
for all $a$ in a neighborhood~$O$ of~$0\in\bbR^s$).  
For the rest of the proof, 
I use this basis to identify~$V$ with~$\bbR^n$. 

Also, by the hypothesis that~$A$ is a Jacobi manifold, 
the linear equations for quantities~$R^\alpha_i$ given as
$$
 {\frac{\partial T^i_{kl}}{\partial a^\alpha}}(a)\,R^\alpha_j
+{\frac{\partial T^i_{lj}}{\partial a^\alpha}}(a)\,R^\alpha_k
+{\frac{\partial T^i_{jk}}{\partial a^\alpha}}(a)\,R^\alpha_l
=\bigl(T^i_{mj}(a)T^m_{kl}(a)
         +T^i_{mk}(a)T^m_{lj}(a)+T^i_{ml}(a)T^m_{jk}(a)\bigr).
$$
are solvable, and the associated homogeneous linear system for
the~$R^\alpha_i$ has, for each value of~$a$, a solution space
of dimension~$s_1+2\,s_2+\cdots+n\,s_n$.  Thus, the equations
are compatible and have constant rank, so there exist real-analytic
functions~$R^\alpha_i(a)$ on a neighborhood of~$0\in\bbR^s$
(which can be supposed to be~$O$)
that furnish solutions to the above inhomogeneous system.

Let~$M = \GL(n,\bbR)\times\bbR^n\times O$, where $O\subset\bbR^s$
is the neighborhood of~$0\in\bbR^s$ selected above.  Let 
$p:M\to\GL(n,\bbR)$, $x:M\to\bbR^n$, and $a:M\to O$ be the respective
projections.  Set~$\eta^i = p^i_j\,\d x^j$ and 
$\pi^\alpha = \d a^\alpha - R^\alpha_i(a)\,\eta^i$.  

Now let~$\cI$ be the ideal on~$M$ generated by the $2$-forms
$$
\Upsilon^i = \d\eta^i +\tfrac12 T^i_{jk}(a)\,\eta^j\w\eta^k.
$$
Note that there exist $1$-forms~$\pi^i_j$ 
such that $\Upsilon^i = \pi^i_j\w\eta^j$ and such that the $1$-forms
$\pi^i_j$, $\eta^i$, and $\pi^\alpha$ are linearly independent
and hence define a coframing on~$M$.

Now, by the way the functions~$R^\alpha_i$ on $O$ were chosen, 
one has
$$
\d\Upsilon^i = 
\tfrac12{\frac{\partial T^i_{jk}}{\partial a^\alpha}}(a)
\,\pi^\alpha\w\eta^j\w\eta^k
+ T^i_{jk}(a)\,\Upsilon^j\w\eta^k\,,
$$
so that~$\cI$ is generated algebraically by the $2$-forms~$\Upsilon^i
= \pi^i_j\w\eta^j$ and the $3$-forms
$$
\Psi^i = {\frac{\partial T^i_{jk}}{\partial a^\alpha}}(a)
              \,\pi^\alpha\w\eta^j\w\eta^k\,.
$$

Since $A$ is involutive, the integral elements in~$\cV_n(\cI)$
defined at each point of~$M$ by~$\pi^i_j=\pi^\alpha=0$ 
are all Cartan-ordinary. By the Cartan-K\"ahler Theorem, 
there is an $n$-dimensional integral manifold~$\cI$ 
tangent to this integral element at the point~$(I_n,0,0)\in M$.

This integral manifold is written as a graph of the form
$\bigl(p^i_j(x),x,a^\alpha(x)\bigr)$ for $x$ in a neighborhood
of~$0\in\bbR^n$.  Now, setting~$\omega^i = p^i_j(x)\,\d x^j$,
one sees that the structure function
of the coframing~$\omega = v_i\omega^i$ is
$$
C = \tfrac12 T^i_{jk}\bigl(a^\alpha(x)\bigr)\,v_i\otimes v^j\w v^k,
$$
which takes values in~$A$ and, in particular, 
takes the value~$a_0\in A$ at $x=0$.
\end{proof}

\begin{remark}[Checking the hypotheses]
Note that, in practical terms, checking the condition 
that~$A\subset V\otimes\Lambda^2(V^*)$ be an involutive Jacobi
manifold can be reduced to a relatively simple calculation:

A coframing satisfying the structure equations~\eqref{eq: domegastreqs}
will necessarily satisfy
$$
0 = \d(\d\omega) 
= -\tfrac12 \d C\w(\omega\w\omega) 
   +\tfrac12 C\bigl(C(\omega\w\omega)\w\omega),
$$
or, relative to a basis~$v_i$ of~$V$ with dual basis~$v^i$,
$$
0 = 
-\tfrac12 \d C^i_{jk}\w\omega^j\w\omega^k
+\tfrac16 (C^i_{mj} C^m_{kl}+C^i_{mk} C^m_{lj}+C^i_{ml} C^m_{jk})
            \,\omega^j\w\omega^k\w\omega^l.
$$
Regarding the~$v^i$ as linear coordinates on~$V$ and regarding
the $C^i_{jk}=- C^i_{kj}$ as the components of the embedding of~$A$
into~$V\otimes\Lambda^2(V^*)$, one can consider the \emph{algebraic}
ideal~$\cI_A$ generated on~$M = A\times V$ by the $3$-forms
$$
\bar\Psi^i = \tfrac12\d C^i_{jk}\w\d v^j\w\d v^k
 - \tfrac16 (C^i_{mj} C^m_{kl}{+}C^i_{mk} C^m_{lj}{+}C^i_{ml} C^m_{jk})
            \,\d v^j\w\d v^k\w\d v^l.
$$
(N.B.:  Just this once, I do \emph{not} want to consider the 
differential closure of~$\cI_A$.)  

Then $\cI_A$ has an integral element $E$ of dimension~$n$ 
based at~$(a,0)\in A\times V$ on which the $\d v^i$ are independent
if and only if~\eqref{eq: Jacobicond} is satisfied.  Moreover,
this integral element is Cartan-ordinary if and only if~$T_aA$
is an involutive subspace of~$V\otimes\Lambda^2(V^*)$.
\end{remark}

\begin{remark}
It will turn out that the $s_i$ for a involutive Jacobi manifold~$A$
have a significance for describing the differential invariants
of $V$-valued coframings taking values in~$A$.  
As will be shown below, in an appropriate sense,
the $V$-valued coframings whose structure functions take values in~$A$
depend (modulo diffeomorphism) on $s_1$ functions of $1$ variable, 
$s_2$ functions of $2$ variables, etc. 
\end{remark}

\section{Ordinary prolongation}
It is time to take a closer look at the geometry of~$\cV^o_n(\cI)$.

\subsection{The tableau of an ordinary element}
Recall that the basepoint projection~$\pi:\cV^o_n(\cI)\to M$ 
is a smooth submersion, so the fiber over~$x$, 
which is $\cV^o_n(\cI)\cap\Gr_n(T_xM)$, 
is a smooth submanifold of~$\Gr_n(T_xM)$. 
For a given~$E\in\cV^o_n(\cI)$, 
the tangent space to this fiber is an involutive tableau
$$
A_E\subset T_E\Gr_n(T_xM)\simeq \bigl(T_xM/E\bigr)\otimes E^*
$$
of dimension~$s_1(E)+2s_2(E)+\cdots+ns_n(E)$, 
and its Cartan characters are given by
$$
s_i(A_E) = s_i(E) + s_{i+1}(E) + \cdots + s_n(E).
$$ 

\subsection{The ordinary prolongation of~$\cI$}
Set~$M^{(1)} = \cV^o_n(\cI)$. Define a subbundle~$C\subset T^*M^{(1)}$ 
by letting
$$
C(E) = \pi^*(E^\perp),
$$
where $E^\perp\subset T^*_{\pi(E)}M$ 
is the annihilator of~$E\subset T_{\pi(E)}M$.  
This subbundle of rank $\dim M - n$ 
is known as the \emph{contact bundle} on $M^{(1)}$.

Let $\cI^{(1)}\subset\cA^*\bigl(M^{(1)}\bigr)$ 
denote the differential ideal generated by the sections of~$C$.  
The ideal~$\cI^{(1)}$ on~$M^{(1)}$
is known as the \emph{ordinary prolongation} of~$\cI$ on~$M$.
(Technically, the definition of the prolongation 
depends on the choice of~$n$, but, in nearly all applications, 
the choice of~$n$ is determined by the problem that~$\cI$ was designed 
to study, so I will not make this part of the notation.)

Every Cartan-ordinary integral manifold~$f:N\to M$ 
has a canonical lift~$f^{(1)}:N\to M^{(1)}$, 
defined by $f^{(1)}(x) = f'(T_xN)\in \cV^o_n(\cI) = M^{(1)}$.
It follows directly from the definition that $f^{(1)}:N\to M^{(1)}$
is an integral manifold of~$\cI^{(1)}$ 
and, moreover, any integral manifold~$h:N\to M^{(1)}$
that is an integral of~$\cI^{(1)}$ 
and has the property that $\pi\circ h:N\to M$ is an immersion 
is of the form~$h=f^{(1)}$, in fact, with $f=\pi\circ h$.

At the integral element level, 
every~$\tilde E\in\cV_n\bigl(\cI^{(1)}\bigr)$ 
with $\tilde E\subset T_EM^{(1)}$ 
such that $\pi':\tilde E\to T_{\pi(E)}M$ is injective 
actually satisfies $\pi'(\tilde E) = E$.
Moreover, each such $\tilde E$ is Cartan-ordinary, 
with Cartan characters
$$
s_i(\tilde E) = s_i(E) + s_{i+1}(E) + \cdots + s_n(E),
$$
and with a flag~$\tilde F=(\tilde E_0,\ldots,\tilde E_{n-1})$ of~$\tilde E$ 
being regular if and only if the flag~$F= (E_0,\ldots,E_{n-1})$
with $E_i=\pi'(\tilde E_i)$ is a regular flag of~$E$.

\subsection{The higher prolongations}
In particular, one can repeat the prolongation process,
but, now considering $M^{(2)}\subset \cV_n^o\bigl(\cI^{(1)}\bigr)$
to be the open subset consisting of those $\tilde E$ that satisfy
the `transversality' condition $\pi'(\tilde E) = E$ (and retaining
the corresponding condition for all the higher prolongations, etc).
This defines a sequence of manifolds~$M^{(k)}$ with ideals~$\cI^{(k)}$,
such that $\bigl(M^{(0)},\cI^{(0)}\bigr)=(M,\cI)$ 
while, for $k\ge1$, the manifold~$M^{(k)}$ 
is embedded as an open subset of~$\cV_n^o\bigl(\cI^{(k-1)}\bigr)$. 
By induction, one sees that the ideal~$\cI^{(k)}$ has Cartan characters
$$
s^{(k)}_j = s_j + {k\choose1}s_{j+1} + {{k+1}\choose2}s_{j+2} 
             +\cdots+ {{k+n-j-1}\choose{n-j}}s_n\,.
$$
One should think of~$M^{(k)}$ 
as the space of $k$-jets of $n$-dimensional
Cartan-ordinary integral manifolds of~$\cI$ in the sense 
that two Cartan-ordinary integral manifolds~$f:N\to M$ and $g:N\to M$
represent the same $k$-jet of an integral manifold at~$x\in N$
if and only if $f^{(k)}(x) = (g\circ h)^{(k)}(x)$ 
for some diffeomorphism~$h:N\to N$ such that $h(x)=x$.

Note that
$$
\dim M^{(k)} = n + {{k}\choose0}s_0 + {{k+1}\choose1}s_1 
              + {{k+2}\choose2}s_2  +\cdots+ {{k+n}\choose{n}}s_n\,,
$$
which is what one would expect for a `solution space' that depends
on $s_0$ constants, $s_1$ functions of $1$ variable, 
$s_2$ functions of $2$ variables, 
$\ldots$, and $s_n$ functions of $n$ variables.

\subsection{Prolonging Cartan structure equations}
This idea can also be applied to understanding 
the differential invariants of the solutions 
to a system of Cartan structure equations such as~\eqref{eq: CSE2}.  
Starting with these equations, 
one can augment them with a system for the~$b^\rho$, namely
\begin{equation}\label{eq: CSE2p}
\d b^\rho = \bigl(G^\rho_i(a,b) + H^\rho_{i\tau}(a,b)c^\tau\bigr)\,\omega^i
\end{equation}
where the the functions~$H^\rho_{i\tau}$ 
for $1\le\tau\le\dim A(a,b)^{(1)}$
are a basis for the first prolongation space of the tableau~$A(a,b)$,
i.e., they give a basis for the solutions of the homogeneous equations
$$
\frac{\partial F^\alpha_i}{\partial b^\rho}(a,b) h^\rho_j
- \frac{\partial F^\alpha_j}{\partial b^\rho}(a,b) h^\rho_i = 0.
$$ 

Using Cartan's ideas, it is not difficult to show that, 
if the system~\eqref{eq: CSE2}
satisfies the hypotheses of Theorem~\ref{thm: CartanThmVar}, 
then the prolonged system of structure equations 
consisting of~\eqref{eq: CSE2} 
and~\eqref{eq: CSE2p} will also satisfy the hypotheses 
of Theorem~\ref{thm: CartanThmVar} 
and that the Cartan characters 
of the tableau of the prolonged system will be
$$
s^{(1)}_i = s_i + s_{i+1} + \cdots + s_n\,.
$$
In particular, in this case, for any given~$(a_0,b_0,c_0)$
there will exist an augmented coframing~$(a,b,c,\omega)$ 
satisfying the prolonged structure equations for which~$(a,b,c)$
assumes the value~$(a_0,b_0,c_0)$.

This leads naturally to the notion of `differential invariants' 
for distinguishing augmented coframings $(a,b,\omega)$ 
satisfying~\eqref{eq: CSE2} up to diffeomorphism.  
Recall that two such coframings~$(a,b,\omega)$ on~$M^n$ 
and~$(\bar a, \bar b,\bar\omega)$ on~$\bar M^n$ are equivalent
up to diffeomorphism if there exists a diffeomorphism~$h:\bar M\to M$
satisfying~$(\bar a,\bar b,\bar\omega)= h^*(a,b,\omega)$. Obviously,
this will imply that, if $\d b^\rho = b^\rho_i\,\omega^i$ 
and $\d {\bar b}^\rho = {\bar b}^\rho_i\,{\bar\omega}^i$, 
then ${\bar b}^\rho_i = h^*( b^\rho_i)$ 
and similarly for all of the derivatives of the $b^\rho_j$ 
expanded in terms of the $\omega^i$.  

Following Cartan's terminology,
one often speaks of the $a^\alpha$ as the \emph{primary} 
(or \emph{fundamental}) invariants of the augmented coframing
and the $b^\rho$ and $b^\rho_i$, etc. as \emph{derived} invariants. 
(Here `invariant' means `invariant under diffeomorphism equivalence'.)

Thus, the import of Theorem~\ref{thm: CartanThmVar} 
is that one sees that, in addition to being able to freely specify 
the values of the $s$ primary invariants (i.e., the $a^\alpha$)
of an augmented coframing~$(a,b,\omega)$ satisfying~\eqref{eq: CSE2} 
at a point, one can also freely specify 
their first derived invariants (i.e., the $b^\rho$), 
which are $r = s_1+s_2+\cdots+s_n$ in number, at the point,
and freely specify 
a certain number of second derived invariants (i.e., the~$c^\tau$)
which are $r^{(1)} = s_1 + 2s_2 + \cdots + ns_n$ in number, at the point,
and so on. 

Applying prolongations successively, 
one sees that the number of freely specifiable differential invariants 
of augmented coframings satisfying~\eqref{eq: CSE2} 
of derived order less than or equal to~$k$ is equal to
$$
s + {k\choose1}s_{1} + {{k+1}\choose2}s_{2} 
             +\cdots+ {{k+n-1}\choose{n}}s_n\,.
$$ 
In a sense that can be made precise, 
this is the dimension of the space of $k$-jets of diffeomorphism
equivalence classes of augmented coframings satisfying~\eqref{eq: CSE2}.

It is in this sense that one can assert 
that, up to diffeomorphism, the `general' augmented coframing 
satisfying a given involutive system of Cartan structure equations 
depends on $s_1$ functions of $1$ variable, 
$s_2$ functions of $2$ variables, and so on.

Similar remarks apply to the structure equations 
of Theorem~\ref{thm: CartanThmVar2}.  In fact, the first prolongation
of these structure equations yield structure equations 
to which Theorem~\ref{thm: CartanThmVar} applies, 
so that one could have simply quoted Theorem~\ref{thm: CartanThmVar} 
to prove Theorem~\ref{thm: CartanThmVar2}. 
This may make the reader wonder why this latter theorem is useful. 
The reason is this:  It is often simpler 
to check the hypotheses of Theorem~\ref{thm: CartanThmVar2} 
for a given set of structure equations
than it is to check the hypotheses of Theorem~\ref{thm: CartanThmVar}
for the prolonged set of structure equations 
(as the reader will see in the examples).

\subsection{Non-ordinary prolongation and the Cartan-Kuranishi Theorem}
In most cases, $\cV_n(\cI)$ does not consist entirely 
of Cartan-ordinary integral elements, and even when 
the open subset $\cV^o_n(\cI)\subset\cV_n(\cI)$ is not empty, 
one is often interested in at least some components of the complement 
and would like to know when there exist integral manifolds 
tangent to these non-ordinary integral elements.  

Cartan's prescription for treating this situation 
was to prolong the non-ordinary integral elements as well:  
Let $M^{(1)}\subset \cV_n(\cI)$ be any submanifold of~$\cV_n(\cI)$
(in most applications, it will be a component of a smooth stratum 
of~$\cV_n(\cI)$ that does not lie in~$\cV^o_n(\cI)$).  
Then, again, one can construct the ideal~$\cI^{(1)}$ 
generated by the sections of the contact subbundle~$C\subset T^*M^{(1)}$ 
and one can consider~$\cV_n\bigl(\cI^{(1)}\bigr)$, 
looking for Cartan-ordinary integral elements of this ideal 
whose projections to~$M$ are injective.
If one finds them, then one has existence 
for integral manifolds tangent to these non-regular integral elements.  
If one does not find them, one can continue the prolongation process
as long as it results in ideals that have integral elements.

Cartan believed that continuing this process 
would always eventually result in either an ideal 
with no integral elements of dimension~$n$
or else one that had Cartan-ordinary integral elements.  
He was never actually able to prove this result, though.  
Finally, a version of this `prolongation theorem' 
was proved by Kuranishi (in the real analytic category, of course).  

The hypotheses of the Cartan-Kuranishi Prolongation Theorem 
are somewhat technical, so I refer you to Kuranishi's original 
paper~\cite{Kur} for those.  In practice, though, 
one uses the Prolongation Theorem as a justification 
for computing successively higher prolongations 
until one reaches involutivity 
(i.e., the existence of Cartan-ordinary integral elements), 
which, nearly always, is what one must do anyway in order 
to prove existence of solutions via Cartan-K\"ahler.

\section{Some applications}

There are many applications of these structure theorems in differential
geometry.  Here is a sample of such applications meant to give the reader
a sense of how they are used in practice.  For further applications 
to differential geometry, the reader can hardly do better than to consult
Cartan's own beautiful collection of instructive examples~\cite{Cartan1945}.

\subsection{Surface metrics with $|\nabla K|^2=1$}
Consider the metrics whose Gauss curvature satisfies~$|\nabla K|^2=1$.
The structure equations are
$$
\begin{aligned}
\d\omega_1 &= -\omega_{12}\w\omega_2\\
\d\omega_2 &= \phm\omega_{12}\w\omega_1\\
\d\omega_{12} &= K\,\omega_1\w\omega_2
\end{aligned}
\qquad\qquad \omega_1\w\omega_2\w\omega_{12} \not=0,
$$ 
where
$$
\d K = \cos b\,\omega_1 + \sin b\,\omega_2\,.
$$
for some function~$b$.  (Here, $b$ is the `free derivative'.) 

Now $\d^2=0$ is an identity for the forms 
in the coframing~$\omega=(\omega_1,\omega_2,\omega_{12})$, while 
$$
0 = \d(\d K) = (\d b - \omega_{12})\w(-\sin b\,\omega_1 + \cos b\,\omega_2).
$$
It follows that the hypotheses of Theorem~\ref{thm: CartanThmVar} 
are satisfied, with the characters of the tableau of free derivatives being
$s_1 = 1$, $s_2 = s_3 = 0$.  Thus, the general (local) solution depends on 
one function of one variable.

The prolonged system will have
$$
\d b = \omega_{12} + c\,(-\sin b\,\omega_1 + \cos b\,\omega_2)
$$
where $c$ is now the new `free derivative', etc.

(Of course, it is not difficult to integrate the structure equations
in this simple case and find an explicit normal form involving one 
arbitrary function of one variable, but I will leave this to the reader.)

\subsection{Surface metrics of Hessian type}
Now, an application of Cartan's original theorem.
The goal is to study those Riemannian surfaces~$(M^2,g)$ 
whose Gauss curvature~$K$ satisfies the second order system
$$
{\mathrm{Hess}}_g(K) = a(K) g + b(K) \d K^2
$$
for some functions~$a$ and~$b$ of one variable.  

Writing~$g = {\omega_1}^2+{\omega_2}^2$ 
on the orthonormal frame bundle~$F^3$ of~$M$, 
the structure equations become
$$
\begin{aligned}
\d\omega_1 &= -\omega_{12}\w\omega_2\\
\d\omega_2 &= \phm\omega_{12}\w\omega_1
\end{aligned}
\qquad\qquad
\begin{aligned}
\d\omega_{12} &= K\,\omega_1\w\omega_2\\
\d K &= K_1\,\omega_1 + K_2\,\omega_2
\end{aligned}
$$
and the condition to be studied is encoded as
$$
\begin{pmatrix}
\d K_1\\ \d K_2
\end{pmatrix}
= \begin{pmatrix}
-K_2\\ \phm K_1
\end{pmatrix}\omega_{12}
+\begin{pmatrix}
a(K)+b(K)\,{K_1}^2 & b(K)\,K_1K_2\\ b(K)\,K_1K_2 & a(K)+b(K)\,{K_2}^2\\
\end{pmatrix}
\begin{pmatrix}
\omega_1\\ \omega_2
\end{pmatrix}.
$$
Applying $\d^2=0$ to these two equations yields
$$
\bigl(a'(K)-a(K)b(K) + K\bigr)\,K_i = 0\qquad \text{for $i=1,2$}.
$$
Thus, unless~$a'(K)=a(K)b(K){-}K$, such metrics have~$K$ constant.

Conversely, suppose that $a'(K)=a(K)b(K){-}K$.
The question becomes `Does there exist a `solution'~$(F^3,\omega)$ 
to the following system?'
$$
\begin{aligned}
\d\omega_1 &= -\omega_{12}\w\omega_2\\
\d\omega_2 &= \phm\omega_{12}\w\omega_1\\
\d\omega_{12} &= K\,\omega_1\w\omega_2
\end{aligned}
\qquad\qquad \omega_1\w\omega_2\w\omega_{12} \not=0,
$$ 
where
$$
\begin{pmatrix}
\d K\\ \d K_1\\ \d K_2
\end{pmatrix}
= 
\begin{pmatrix}
K_1 & K_2  & 0\\
a(K)+b(K)\,{K_1}^2 & b(K)\,K_1K_2 & -K_2 \\ 
b(K)\,K_1K_2 & a(K)+b(K)\,{K_2}^2 & \phm K_1\\
\end{pmatrix}
\begin{pmatrix}
\omega_1\\ \omega_2\\ \omega_{12}
\end{pmatrix}.
$$

Since $\d^2=0$ is formally satisfied for these structure equations,
Theorem~\ref{thm: CTFT} applies and guarantees that, 
for any constants~$(k,k_1,k_2)$, 
there is a local solution with the invariants
$(K,K_1,K_2)$ taking the value $(k,k_1,k_2)$.

In fact, the above equations show that, on a solution, the $\bbR^3$-valued 
function~$(K,K_1,K_2)$ either has rank $0$ (if $K_1=K_2=a(K)=0$) or
rank~$2$.  Moreover, one sees that
$$
-\bigl(a(K)+b(K)({K_1}^2{+}{K_2}^2)\bigr)\,\d K
    +K_1\,\d K_1+K_2\,\d K_2 = 0,
$$
so that the image of a connected solution lies in an integral leaf 
of this $1$-form, which only vanishes when $K_1=K_2=a(K)=0$.  
Setting $L={K_1}^2{+}{K_2}^2$, this expression becomes
$$
-2\bigl(a(K)+b(K)L\bigr)\,\d K+\d L = 0,
$$
which has an integrating factor:  If $\lambda(K)$ is a
nonzero solution to $\lambda'(K) = -b(K)\lambda(K)$, then
$$
-2\lambda(K)^2a(K)\,\d K + \d\bigl(\lambda(K)^2L) = 0,
$$
so that the curvature map has image in a level set
of the function $F(K,K_1,K_2) = \lambda(K)^2({K_1}^2{+}{K_2}^2) - \mu(K)$, 
where $\mu'(K) = 2\lambda(K)^2a(K)$.  (This function has critical points
only where $K_1=K_2=a(K)=0$.)

On any solution~$(F^3,\omega)$, 
the vector field $Y$ defined by the equations
$$
\omega_1(Y) =  \lambda(K)K_2,\quad
\omega_2(Y) = -\lambda(K)K_1,\quad
\omega_{12}(Y) =  \lambda(K)a(K),
$$
is a symmetry vector field of the coframing (since the Lie derivative
of each of $\omega_1$, $\omega_2$, $\omega_{12}$ with respect to~$Y$ 
is zero).  It is nonvanishing on a solution of rank~$2$, and, 
up to constant multiples, it is the unique symmetry vector field 
of the coframing on any connected solution.

For simplicity, I will only consider the case~$b(K)\equiv0$ 
in the remainder of this discussion.  In this case,~$a'(K)=-K$, 
so $a(K) = \tfrac12(C-K^2)$ for some constant~$C$ 
and~$\lambda'(K)=0$, so one can take~$\lambda(K)\equiv 1$.

The most interesting case is when $C>0$, and, by scaling the metric~$g$
by a constant, one can reduce to the case $C=1$.  
Thus, the equations simplify to  
$$
\begin{pmatrix}
\d K\\ \d K_1\\ \d K_2
\end{pmatrix}
= 
\begin{pmatrix}
K_1 & K_2  & 0\\
\tfrac12(1{-}K^2) & 0 & -K_2 \\ 
0 & \tfrac12(1{-}K^2) & \phm K_1\\
\end{pmatrix}
\begin{pmatrix}
\omega_1\\ \omega_2\\ \omega_{12}
\end{pmatrix}.
$$
and these functions satisfy
$$
F(K,K_1,K_2) = {K_1}^2{+}{K_2}^2 + \tfrac13 K^3 - K = C
$$
where $C$ is a constant 
(different from the previous $C$, which is now normalized to~$1$).

There are two critical points of~$F$, namely $(K,K_1,K_2) = (\pm1,0,0)$,
and these correspond to the surfaces whose Gauss curvature 
is identically $+1$ or identically $-1$.  These clearly exist globally
so it remains to consider the other level sets.  

The level sets with $C < -\tfrac23$ are connected and contractible,
in fact, they can be written as graphs of $K$ as a function 
of~${K_1}^2{+}{K_2}^2$.  $C=-\tfrac23$ contains the critical 
point $(K,K_1,K_2) = (1,0,0)$, but away from this point, it is also
a smooth graph.  When $-\tfrac23<C<\tfrac23$, the level set has
two smooth components, a compact $2$-sphere that encloses the 
critical point $(1,0,0)$ and a graph of $K$ as a smooth function 
of~${K_1}^2{+}{K_2}^2$.  The level set~$C=\tfrac23$ is singular
at the point $(-1,0,0)$, but, minus this point, it has two smooth
pieces, one bounded and simply connected, and one unbounded and 
diffeomorphic to $\bbR\times S^1$.  For $C>\tfrac23$, the level
set is connected and contractible.

According to the general theory, for each contractible
component~$L$ of a (smooth part of a) level set~$F=C$, 
there will exist a simply-connected solution manifold~$(F^3,\omega)$ 
whose curvature image is~$L$ and whose symmetry vector field~$Y$
is complete.  Moreover, the time-$2\pi$-flow of the vector field
$X_{12}$ (i.e., the vector field that satisfies~$\omega_1(X_{12})
=\omega_2(X_{12})=0$ while $\omega_{12}(X_{12})=1$) is a symmetry
of the coframing~$\omega$ and hence is the time-$T$-flow of~$Y$
for some~$T>0$.  Dividing~$F$ by the $\bbZ$-action that this 
generates produces a solution manifold~$(\bar F,\omega)$ 
that is no longer simply-connected 
but on which the flow of~$X_{12}$ is $2\pi$-periodic,
and this is the necessary and sufficient condition that $\bar F$
be the oriented orthonormal frame bundle of a Riemannian
surface~$(M^2,g)$ satisfying the desired equation.

However, for the components of the level sets that are diffeomorphic 
to the $2$-sphere, this global existence result does not generally
hold, i.e., the corresponding solution manifold~$(F^3,\omega)$ 
need not be the orthonormal frame bundles of complete 
Riemannian surfaces~$(M^2,g)$. I will explain why for the 
$2$-sphere components of the level sets $F=\epsilon^2-2/3$
where $\epsilon>0$ is small.  

Suppose that a connected solution manifold~$(F^3,\omega)$
whose curvature map has, as image, such a $2$-sphere component
is found and that the symmetry vector field~$Y$ as defined
above is complete on it. 
Then the metric~$h = {\omega_1}^2+{\omega_2}^2+{\omega_{12}}^2$ 
must be complete on~$F$.  Now, for small positive $\epsilon$, 
one has that $K$ is close to~$1$ while $K_1$ and $K_2$ are 
close to zero, so it follows from a computation that the 
sectional curvatures of~$h$ are all positive.  In particular,
the completeness of the metric on~$F^3$ implies, by Bonnet-Meyers,
that it is compact, with finite fundamental group.  

By passing to a finite cover, one can assume that $F$ 
is simply connected.  I claim that the symmetry vector field~$Y$ 
has closed orbits and that its flow generates an~$S^1$-action on~$F$.  
To see this, note that the map~$(K,K_1,K_2):F\to\bbR^3$ submerses 
onto the $2$-sphere leaf.  Hence the fibers over the two points 
where $K_1=K_2=0$ must be a finite collection of circles that are 
necessarily integral curves of the vector field~$Y$, which has
no singular points.  In particular, the flow of~$Y$ on one of these
circles must be periodic, but, because the flow of~$Y$ preserves
the coframing~$\omega$, if some time~$T>0$ flow of~$Y$ has a fixed
point, then the time~$T$ flow of~$Y$ must be the identity.  
Thus, the flow of~$Y$ is periodic with some minimal positive period~$T>0$, 
so it generates a free $S^1$-action on~$F$. 
The quotient by this free~$S^1$-action is a connected quotient
surface that is a covering of the $2$-sphere.  
Since this covering must be trivial, the orbits of~$Y$ 
are the fibers of the map~$(K,K_1,K_2)$ to the $2$-sphere.
In particular, $F$, being connected and simply-connected, 
must be diffeomorphic to the $3$-sphere.  

Now, consider the vector field~$X_{12}$ on~$F$ as defined above.
This vector field is $(K,K_1,K_2)$-related to the vector field
$$
-K_2\,\frac{\partial\hfill}{\partial K_1}
 +K_1\,\frac{\partial\hfill}{\partial K_2}
$$
on~$\bbR^3$ whose flow is rotation about the $K$-axis with period $2\pi$.
 
It also follows that the flow of $X_{12}$ 
preserves the two circles that are defined by~$K_1=K_2=0$.
If~$(F,\omega)$ is to be a covering of the orthonormal frame bundle 
of a Riemannian surface~$(M^2,g)$, then $X_{12}$ must be periodic
of period~$2k\pi$ for some integer~$k>0$.
As already remarked, by the structure equations, 
the $2\pi$-flow of~$X_{12}$, say~$\Psi$,
is a symmetry of the coframing and hence must be the time~$R>0$
flow of~$Y$ for some unique~$R\in(0,T]$.

Now, along each of the two circles in~$F$ defined by $K_1=K_2=0$, 
one has $Y = a(K)X_{12}\not=0$.
The two points where~$K_1=K_2=0$ satisfy~$K=K_\pm(\epsilon)$ 
where $K_-(\epsilon)<1< K_+(\epsilon)$ 
and $\tfrac13 K_\pm(\epsilon)^3 - K_\pm(\epsilon) = \epsilon^2- \frac23$.
In fact, one finds expansions
$$
K_\pm(\epsilon) = 
1 \pm \epsilon - \tfrac16\epsilon^2 \pm \tfrac{5}{72} \epsilon^3 -\cdots
$$
and this implies that 
$$
a\bigl(K_\pm(\epsilon)\bigr) = \tfrac12(1-K_\pm(\epsilon)^2)
= \mp \epsilon -\tfrac13 \epsilon^2 + \cdots.
$$  
Thus the ratios of $X_{12}$ to $Y$ on these two circles
are not equal or opposite, and hence $Y$ cannot have the same period
on these two circles, which is impossible.  Thus, there
cannot be a global solution surface for such a leaf.

\subsection{Prescribed curvature equations for Finsler surfaces}
For an oriented Finsler surface $(M^2,F)$, Cartan showed that the `tangent 
indicatrix' (i.e., the analog of the unit sphere bundle)~$\Sigma\subset TM$
carries a canonical coframing~$(\omega_1,\omega_2,\omega_3)$ generalizing
the case of the unit sphere bundle of a Riemannian metric.  It satisfies structure equations
\begin{equation}\label{eq: FinslerSE}
\begin{aligned}
\d\omega_1 &= -\omega_2\w\omega_3\\
\d\omega_2 &= -\omega_3\w\omega_1 - I\,\omega_2\w\omega_3\\
\d\omega_3 &= -K\,\omega_1\w\omega_2 - J\,\omega_2\w\omega_3
\end{aligned}
\qquad\qquad \omega_1\w\omega_2\w\omega_3 \not=0,
\end{equation}
where I have written $\omega_3$ for what would be $-\omega_{12}$
in the Riemannian case.  The functions~$I$, $J$, and $K$ are the
\emph{Finsler structure functions}.  

One can check that Theorem~\ref{thm: CartanThmVar2} 
applies directly to these equations, with~$V$ of dimension~$3$
and~$A\subset V\otimes\Lambda^2(V^*)$ 
an affine subspace of dimension~$3$ 
(and on which $I$, $J$, and~$K$ are coordinates).  The Cartan
characters are $s_1=0$, $s_2=2$, and $s_3=1$.
Thus, the general Finsler surface
depends on one function of $3$ variables, which is to be expected, 
since a Finsler structure on $M$ is locally determined by choosing 
a hypersurface in~$TM$ (satisfying certain local convexity conditions) 
to be the tangent indicatrix~$\Sigma\subset TM$.  
In fact, if $\omega=(\omega_i)$ is any coframing 
on a $3$-manifold~$\Sigma^3$ that satisfies~\eqref{eq: FinslerSE} 
such that the space~$M$ of leaves of the system~$\omega_1=\omega_2=0$ 
can be given the structure of a smooth surface 
for which the natural projection~$\pi:\Sigma\to M$ is a submersion, 
then $\Sigma$ has a natural immersion~$\iota:\Sigma\to TM$
defined by letting~$\iota(u)=\pi'\bigl(X_1(u)\bigr)$ for~$u\in\Sigma$,
where~$X_1$ is the vector field on~$\Sigma$ dual to~$\omega_1$, 
and, locally, this defines a Finsler structure on~$M$.

Taking the exterior derivatives of~\eqref{eq: FinslerSE}, 
one finds that they satisfy identities 
(the `Bianchi identities' of Finsler geometry)
\begin{equation}\label{eq: CartanBianchiIds}
\begin{aligned}
\d I ={}&&        J\,\omega_1 &&{}+ I_2\,\omega_2 &&{}+ I_3\,\omega_3\,,&{}\\
\d J ={}&&-(K_3+KI)\,\omega_1 &&{}+ J_2\,\omega_2 &&{}+ J_3\,\omega_3\,,&{}\\
\d K ={}&&      K_1\,\omega_1 &&{}+ K_2\,\omega_2 &&{}+ K_3\,\omega_3\,.&{}\\\end{aligned}
\end{equation}
for seven new functions~$I_2$, $I_3$, $\ldots$, $K_3$.  
These are the free derivatives of the structure theory.
As expected from the general theory, the tableau of free derivatives 
of the prolonged system, i.e., \eqref{eq: FinslerSE} 
together with \eqref{eq: CartanBianchiIds}, 
is involutive with characters $s_1=s_2=3$ and $s_3=1$.  

Now, by the structure equations~\eqref{eq: CartanBianchiIds}, 
if $I=0$, then $J=0$ and $K_3 =0$,
so that the Bianchi identities reduce to
$$
\d K = K_1\,\omega_1 + K_2\,\omega_2\,,
$$
which is simply the Riemannian case.  Note that in this case, 
the tableau of free derivatives has $s_1=s_2=1$ while $s_3=0$, 
corresponding to the fact that Riemannian surfaces depend locally 
on one function of $2$ variables (up to diffeomorphism).

One can, of course, study other curvature conditions.  For example,
the \emph{Landsberg} surfaces are those for which $J=0$.  They satisfy
structure equations
\begin{equation}\label{eq: LandsbergBianchiIds}
\begin{aligned}
\d I ={}&&        0\,\omega_1 &&{}+ I_2\,\omega_2 &&{}+ I_3\,\omega_3\,,&{}\\
\d K ={}&&      K_1\,\omega_1 &&{}+ K_2\,\omega_2 &&{}- KI\,\omega_3\,.&{}\\
\end{aligned}
\end{equation}
The tableau of free derivatives now has $s_1 = s_2 = 2$ and $s_3 = 0$,
so that the general Landsberg metric depends on $2$ functions of $2$ 
variables.  (By the way, this is only a `microlocal' description
of the solutions; constructing global solutions is much more difficult. 
However, it does suffice to show how `flexible' 
the `microlocal' solutions are.)

Another common curvature condition is the `$K$-basic' condition, i.e.,
when, $K$, the Finsler-Gauss curvature, is constant on the fibers of
the projection~$\Sigma\to M$.  This is the condition~$K_3 = 0$, so that
the structure equations become
\begin{equation}\label{eq: KbasicBianchiIds}
\begin{aligned}
\d I ={}&&   J\,\omega_1 &&{}+ I_2\,\omega_2 &&{}+ I_3\,\omega_3\,,&{}\\
\d J ={}&& -KI\,\omega_1 &&{}+ J_2\,\omega_2 &&{}+ J_3\,\omega_3\,,&{}\\
\d K ={}&& K_1\,\omega_1 &&{}+ K_2\,\omega_2 &&{}+ 0\,\omega_3\,.&{}\\
\end{aligned}
\end{equation}
The tableau of free derivatives now has $s_1=s_2=3$ and $s_3=0$, showing
that these Finsler structures depend on $3$ functions of $2$ variables.

Even more restrictive are the Finsler metrics with constant~$K$.  
These satisfy
\begin{equation}\label{eq: KconstBianchiIds}
\begin{aligned}
\d I ={}&&   J\,\omega_1 &&{}+ I_2\,\omega_2 &&{}+ I_3\,\omega_3\,,&{}\\
\d J ={}&& -KI\,\omega_1 &&{}+ J_2\,\omega_2 &&{}+ J_3\,\omega_3\,,&{}\\
\d K ={}&&   0\,\omega_1 &&{}+ 0\,\omega_2 &&{}+ 0\,\omega_3\,.&{}\\
\end{aligned}
\end{equation}
The tableau of free derivatives now has $s_1=s_2=2$ and $s_3=0$, showing
that these Finsler structures depend on $2$ functions of $2$ variables.
(For those who know about characteristics, note that, in this case,
a covector $\xi = \xi_1\,\omega_1 + \xi_2\,\omega_2 + \xi_3\,\omega_3$
is characteristic for this tableau if and only if $\xi_1 = 0$.  Thus,
the `arbitrary functions' are actually functions on the leaf space
of the geodesic flow~$\omega_2=\omega_3=0$.  This suggests (and, of course, it turns out to be true) that these structures are actually geometric structures on the space of geodesics in disguise.)

\subsection{Ricci-gradient metrics in dimension~$3$}
Here are some sample problems from Riemannian geometry.
In the following, for simplicity of notation, 
I will consider only the $3$-dimensional case, 
but the higher dimensional cases are not much different.

Consider the problem of studying those Riemannian manifolds~$(M,g)$
for which there exists a function~$f$ such that $\mathrm{Ric}(g)
= (\d f)^2 + H(f)\,g$, where $H$ is a specified function of one variable.
Most metrics~$g$ will not have such a `Ricci potential', 
and it is not clear how many such metrics there are.  

The problem can be set up in structure equations as follows:  
On the orthonormal frame bundle~$F^6\to M^3$ of~$g$, 
one has the usual first structure equations
\begin{equation}\label{eq: SE1_Riccigrad}
\d\omega_i = -\omega_{ij}\w\omega_j
\end{equation}
and the second structure equations (in dimension $3$) can be written
in the form
\begin{equation}\label{eq: SE2_Riccigrad}
\begin{pmatrix}
\d\omega_{23}\\ \d\omega_{31}\\ \d\omega_{12}
  \end{pmatrix}
= 
- \begin{pmatrix}
\omega_{12}\w\omega_{31}\\ \omega_{23}\w\omega_{12}\\ 
\omega_{31}\w\omega_{23}
  \end{pmatrix}
- \left(R -\tfrac12{\mathrm{tr}}(R)\,I_3\right)
\begin{pmatrix}
\omega_2\w\omega_3\\ \omega_3\w\omega_1\\ \omega_1\w\omega_2
  \end{pmatrix}
\end{equation}
where $R=(R_{ij})$ is the symmetric matrix of the Ricci tensor.  
By hypothesis, there exists a function~$f$ such that 
$$
R_{ij} = f_if_j + H(f)\delta_{ij}
$$
where 
\begin{equation}\label{eq: SE3_Riccigrad}
\d f = f_1\,\omega_1 + f_2\,\omega_2 + f_3\,\omega_3\,.
\end{equation}
The four functions~$(f,f_1,f_2,f_3)$ 
will play the role of the $a^\alpha$ in the structure equations.  
Since $\d(\d f)=0$, there exist functions~$f_{ij}=f_{ji}$ such that
\begin{equation}\label{eq: SE3.5_Riccigrad}
\d f_i = -\omega_{ij} f_j + f_{ij}\,\omega_j\,.
\end{equation}
The symmetry of~$R$ implies that the equations~$\d(\d\omega_i)=0$ 
are identities, but, when one computes~$\d(\d\omega_{ij})=0$, 
one finds that these relations can be written as
$$
\bigl(2(f_{11}+f_{22}+f_{33}) - H'(f)\bigr) \,\d f = 0.
$$
Thus, either $\d f =0$, in which case $f$ is constant 
(so that the metric is Einstein), or else the relation
$$
f_{11}+f_{22}+f_{33} - \tfrac12 H'(f) = 0
$$
must hold.  So impose this condition, 
and rewrite the above equation in the form
\begin{equation}\label{eq: SE4_Riccigrad}
\d f_i = -\omega_{ij} f_j
            + \bigl(b_{ij}+ \tfrac16H'(f)\delta_{ij}\bigr)\,\omega_j\,.
\end{equation}
where the (new) $b_{ij}=b_{ji}$ are subject to the trace condition
$b_{11}+b_{22}+b_{33}=0$.  These~$b_{ij}$ will play the role 
of the~$b^\rho$ in the structure equations.  

Thus, the problem can be thought of as seeking 
coframings~$\omega= (\omega_i,\omega_{ij})$ and functions~$(f,f_i)$ 
on a $6$-manifold~$F^6$ that satisfy the equations
\eqref{eq: SE1_Riccigrad}, \eqref{eq: SE2_Riccigrad}, 
\eqref{eq: SE3_Riccigrad}, and \eqref{eq: SE4_Riccigrad},
where the $b_{ij}=b_{ji}$ are subject to~$b_{11}+b_{22}+b_{33}=0$.

The tableau of the free derivatives is involutive, 
with characters~$s_1=3$, $s_2=2$, and $s_k=0$ for~$3\le k\le 6$.
Moreover, the equations $\d(\d\omega_i)=\d(\d\omega_{ij})=\d(\d f) = 0$
are identities while the equations~$\d(\d f_i) = 0$ 
are satisfiable in the form
$$
\d b_{ij} = -b_{ik}\omega_{kj}-b_{kj}\omega_{ki} 
+ F\,(3f_i\omega_j+3f_j\omega_i-2\delta_{ij}f_k\omega_k)
+ b_{ijk}\omega_k
$$
where $F=\tfrac1{10}\bigl({f_1}^2{+}{f_2}^2{+}{f_3}^2{+}H(f)+\tfrac13H''(f)\bigr)$ and where $b_{ijk}=b_{jik}=b_{ikj}$ 
and $b_{iik} = 0$.  
Hence, there are $7 = s_1 + 2\,s_2 + \cdots + 6\,s_6$ 
independent free derivatives of the~$b_{ij}$, 
the maximum allowed by the characters of their tableau.

Thus, the hypotheses of Theorem~\ref{thm: CartanThmVar} are satisfied.
Consequently, when~$H$ is an analytic function,
the pairs $(g,f)$ that satisfy~${\mathrm{Ric}}(g) = (\d f)^2 + H(f)\,g$
depend on $2$ functions of $2$ variables (up to diffeomorphism).  

(For those who know about the characteristic variety: 
A nonzero covector~$\xi = \xi_i\,\omega_i+\xi_{ij}\,\omega_{ij}$ 
is characteristic if and only if $\xi_{ij}=0$ 
and~${\xi_1}^2{+}{\xi_2}^2{+}{\xi_3}^2=0$.  Thus,
the real characteristic variety is empty, so the solutions
are all real analytic when $H$ is real analytic.)

More generally, one can consider 
the problem of studying those Riemannian manifolds~$(M,g)$
for which there exists a function~$f$ such that 
\begin{equation}\label{eq: generalRiccipotential}
\mathrm{Ric}(g) = a(f)\,\Hess_g(f) + b(f)\,(\d f)^2 + c(f)\,g
\end{equation}
where $a$, $b$, and $c$ are specified functions of one variable
and~$\Hess_g(f)=\nabla\nabla f$ 
is the Hessian of $f$ with respect to~$g$, 
i.e., the quadratic form that is the second covariant derivative of~$f$ 
with respect to the Levi-Civita connection of~$g$.
For example, when $a(f)=-1$, $b(f)=0$, and $c(f)=\lambda$ (a constant), 
\eqref{eq: generalRiccipotential} is the equation for a gradient Ricci soliton.
For simplicity, in what follows, I will assume that $a$, $b$, and $c$
are real-analytic functions. 

If $a(f)\equiv b(f)\equiv0$, then \eqref{eq: generalRiccipotential} 
implies that $g$ is an Einstein metric, 
and so the only solutions $(g,f)$ are ones for which $c(f)$ is a constant.  
In particular, if $c'(f)$ is not identically vanishing, 
then the only solutions~$(g,f)$ 
are when $g$ is Einstein and~$f$ is a constant.

If $a(f)\equiv0$ and $b(f)>0$, 
one can reduce~\eqref{eq: generalRiccipotential} 
to the case $b(f)\equiv1$ (which was treated above)
by replacing~$(g,f)$ by~$\bigl(g,\phi(f)\bigr)$, where $\phi'(f)^2=b(f)$.  
(Meanwhile, when $b(f)<0$, one can reduce to $b(f)\equiv-1$ 
by replacing~$(g,f)$ by~$\bigl(g,\phi(f)\bigr)$ where $\phi'(f)^2 = -b(f)$.
The reader can easily check that the local analysis of this case 
is essentially the same as the case $a(f)\equiv0$ and $b(f)\equiv 1$,
with a few sign changes.)

In the `generic' case, in which $a$ is nonvanishing, 
one can reduce to the case $b(f)\equiv0$ 
by replacing $(g,f)$ by~$\bigl(g,\phi(f)\bigr)$ 
where $\phi$ is a function that satisfies~$\phi'(x) >0$
and $\phi''(x)=\bigl(b(x)/a(x)\bigr)\phi'(x)$.  
Hence, I will consider only the case $b(f)\equiv0$ 
in the remainder of this discussion.
 
Thus, the equation to be studied is encoded 
with the same structure equations~\eqref{eq: SE1_Riccigrad},
\eqref{eq: SE2_Riccigrad}, and \eqref{eq: SE3.5_Riccigrad}
but now with the relations
$$
R_{ij} = a(f)\,f_{ij} + c(f)\,\delta_{ij}\,,
$$
where $a$ is a nonvanishing function.
The equations~$\d(\d\omega_{ij})=\d(\d f_i)=0$ 
then turn out to imply the relation
$$
\d\left(\frac{L(f)}{a(f)}\right)
+\left(1+\frac{a'(f)}{a(f)^2}\right)\,\d H(f)
+\frac{\bigl(2a(f)c(f){-}c'(f)\bigr)}{a(f)^2}\,\d f = 0
$$
where $L(f) = f_{11}{+}f_{22}{+}f_{33}$ 
and $H(f) = {f_1}^2{+}{f_2}^2{+}{f_3}^2$.
Taking the exterior derivative of this relation yields
$$
\left(\frac{a(f)a''(f)-2a'(f)^2}{a(f)^3}\right)\, \d f \w \d H(f) = 0.
$$

At this point, the study of these equations 
divides into cases, 
depending on whether $aa''-2(a')^2$ vanishes identically or not.

If the function $aa''-2(a')^2$ does not vanish identically,
then any pair~$(g,f)$ that satisfies the original equation must
also satisfy equations of the form
$$
{f_1}^2{+}{f_2}^2{+}{f_3}^2 = h(f)
$$
and 
$$
f_{11}{+}f_{22}{+}f_{33} = a(f)l(f)
$$
for functions~$l$ and $h$ of a single variable that satisfy
$$
l'(x) +\left(1+\frac{a'(x)}{a(x)^2}\right)\,h'(x)
+\frac{\bigl(2a(x)c(x){-}c'(x)\bigr)}{a(x)^2} = 0.
$$
The first of these equations implies, upon differentiation,
$$
2f_{ij}f_j = h'(f)f_j
$$
which, as long as~$h(f)>0$, gives three equations 
on the free derivatives $f_{ij}=f_{ji}$. 
Moreover, the equation $f_{11}{+}f_{22}{+}f_{33} = a(f)l(f)$ 
is independent from these three.  
This means that there is only a $2$-parameter family 
of possible variation in the~$f_{ij}$.  In fact, the tableau of free
derivatives in this case is involutive with $s_1=2$ and all $s_i=0$
for $i>1$, so that solutions of this system depend on at most%
\footnote{The reason for the `at most' is that I have not verified
that the torsion is absorbable, so I cannot claim that this
prolonged system is involutive.} 
two functions of one variable (three if you count the function $h$).  
Thus, the pairs $(g,f)$ that satisfy the above equation 
are rather rigid.

On the other hand, if $aa''-2(a')^2$ vanishes identically, 
then $a(f) = 1/(c_0+c_1f)$ for some constants $c_0$ and $c_1$,
not both zero.  

If $c_1=0$, then, by scaling $f$, one can reduce
to the case $a(f) = 1$ and the original equation becomes
$$
R_{ij} = f_{ij} + c(f)\,\delta_{ij}\,,
$$
while the relation above becomes
$$
f_{11}+f_{22}+f_{33} + {f_1}^2 + {f_2}^2 + {f_3}^2 - c(f) + 2C(f) = \lambda,
$$
where $C'(f) = c(f)$, and where $\lambda$ is a constant.
Adding this relation on the `free derivatives'~$f_{ij}$ yields
a tableau of free derivatives that has $s_1=3$, $s_2=2$ and $s_j=0$
for $j>2$.  Moreover, a short calculation reveals that this relation
satisfies the conditions of Theorem~\ref{thm: CartanThmVar}, 
so, up to diffeomorphism, the local general pairs $(g,f)$ 
that satisfy a relation of the form $\Ric(g) = \Hess_g(f) + c(f) g$
(for a fixed real-analytic function $c(f)$) depend on two functions
of two variables.  

Meanwhile, if $c_1\not=0$, then by translating and scaling $f$, one
can reduce to the case $a(f) = 1/f$, and one gets a similar result,
that, up to diffeomorphism, the local general pairs $(g,f)$ (with,
say $f>0$) 
that satisfy a relation of the form $\Ric(g) = (\Hess_g(f))/f + c(f) g$
(for a fixed real-analytic function $c(f)$) also depend on two functions
of two variables.  

\subsection{Riemannian $3$-manifolds with constant Ricci eigenvalues}
In dimension~$3$, a different way of writing the structure equations 
on the orthonormal frame bundle~$F^6$ of~$(M^3,g)$ 
is to write them in `vector' form as
\begin{equation}\label{eq: VectStrEq1}
\d\eta = -\theta\w\eta
\end{equation}
and
\begin{equation}\label{eq: VectStrEq2}
\d\theta = -\theta\w\theta 
            + \bigl(R-\tfrac14\tr(R) I_3\bigr)\,\eta\w{}^t\eta
            + \eta\w {}^t\eta\,\bigl(R-\tfrac14\tr(R) I_3\bigr)
\end{equation}
where~$\eta = (\eta_i)$ takes values in~$\bbR^3$ 
(thought of as columns of real numbers of height~$3$) 
and~$\pi^*g = {}^t\eta\circ\eta$,
while~$\theta = -{}^t\theta = (\theta_{ij})$ takes values in~$\euso(3)$, 
the space of skewsymmetric $3$-by-$3$ matrices, 
and $R={}^tR$ is the $3$-by-$3$ symmetric matrix that 
represents the Ricci curvature, i.e., 
$R = (R_{ij})$ and $\pi^*\bigl(\Ric(g)\bigr)=R_{ij}\,\eta_i\circ\eta_j$.

\subsubsection{The general metric}
Setting~$V = \bbR^3\oplus\euso(3)$ (so that, again,~$n=6$), 
then~$\omega = (\eta,\theta)$
is a $V$-valued coframing, and the above structure equations 
take the form~$\d\omega=-\tfrac12C(\omega\w\omega)$,
where $C$ takes values 
in a $6$-dimensional affine subspace~$A\subset V\otimes\Lambda^2(V^*)$.

The exterior derivatives of these structure equations then give
the compatibility conditions:  One has $\d(\d\eta) = 0$, 
and, setting~$\rho = \d R +\theta R - R\theta$, one finds
\begin{equation}
\d(\d\theta) = \bigl(\rho-\tfrac14\tr(\rho) I_3\bigr)\w\eta\w{}^t\eta
            + \eta\w {}^t\eta\w\bigl(\rho-\tfrac14\tr(\rho) I_3\bigr),
\end{equation}
so~$A\subset V\otimes\Lambda^2(V^*)$, an affine subspace, 
is a Jacobi manifold.   Inspection shows that its tableau 
has characters~$s_0=s_1=0$, $s_2=s_3=3$, and $s_k=0$ for~$k=4,5,6$.  
Now, the three $3$-forms $\d(\d\theta)$ 
place $21$ restrictions on the $36$ coefficients of~$S\in\Hom(V,\bbR^6)$ 
in order that the equation $\rho - S(\eta,\theta)=0$ 
should define an integral element.  
Since $21 = c_0+c_1+c_2+c_3+c_4+c_5$, 
it follows that the tableau is involutive, 
so that~$A$ is an involutive Jacobi manifold.

Thus, Theorem~\ref{thm: CartanThmVar2} yields the expected
result that the general metric in dimension $3$ 
modulo diffeomorphism depends on $3$ functions of $2$ variables 
and $3$ functions of $3$ variables.  
Applying prolongation to the structure equations 
would yield that the number of differential invariants 
of the coframing of order at most~$k+1$ is
$$
\sum_{j=0}^6 {{k+j-1}\choose j}\,s_j
= \frac{k(k+1)(k+5)}2\,,
$$
which is the classically known number 
of independent derivatives of the curvature functions $R_{ij}$ 
of order at most $k{-}1$ (as expected, 
since the~$R_{ij}$ themselves are the first derivatives of the 
coframing~$\omega$).

\subsubsection{Constant Ricci eigenvalues}
More interesting are the proper submanifolds of~$A$ 
that are involutive Jacobi manifolds.
For example, suppose that one wanted to determine the generality
(modulo diffeomorphisms) of the space of metrics whose Ricci
tensor has constant eigenvalues.   Thus, one takes the above
structure equations and imposes that
\begin{equation}\label{eq: Rconsteigens}
R = {}^tP C P = P^{-1}CP,
\end{equation}
where $C$ is a constant diagonal matrix
with diagonal entries~$c = (c_1,c_2,c_3)$ where~$c_1\ge c_2\ge c_3$
and $P$ lies in~$\SO(3)$.  Restricting $R$ to take this form
in the structure equations defines a (non-affine) 
submanifold~$B_c\subset A \subset V\otimes\Lambda^2(V^*)$ 
that has dimension~$3$ (and is diffeomorphic to the quotient
of $\SO(3)$ by its diagonal subgroup) when the $c_i$ are distinct,
dimension~$2$ (and is diffeomorphic to $\bbR\bbP^2$) when two of
the $c_i$ are equal, and has dimension~$0$ (and is a single point)
when the $c_i$ are all equal.

One can write the structure equations in a relatively uniform way
by setting $\bar\eta = P\eta$ and $\pi = \d P P^{-1} - P\theta P^{-1}
= -{}^t\pi$, for then the above equations can be written
$$
0 = P\d(\d\theta)P^{-1} 
= (C\pi-\pi C)\w\bar\eta\w{}^t\bar\eta 
+ \bar\eta\w{}^t\bar\eta\w (C\pi-\pi C)          
$$
and the three $3$-forms in the skew-symmetric matrix on the righthand
side of this equation are seen to be
$$
\begin{aligned}
\Upsilon_1 &= \bigl( (c_3{-}c_1)\pi_2\w\bar\eta_2
                    -(c_1{-}c_2)\pi_3\w\bar\eta_3\bigr)\w\bar\eta_1 \\
\Upsilon_2 &= \bigl( (c_1{-}c_2)\pi_3\w\bar\eta_3
                    -(c_2{-}c_3)\pi_1\w\bar\eta_1\bigr)\w\bar\eta_2\\
\Upsilon_3 &= \bigl( (c_2{-}c_3)\pi_1\w\bar\eta_1
                    -(c_3{-}c_1)\pi_2\w\bar\eta_2\bigr)\w\bar\eta_3\\
\end{aligned}
$$
where~$\pi = (\pi_{ij}) = (\epsilon_{ijk}\pi_k)$.  
Note that the $1$-forms $\pi_1,\pi_2,\pi_3$ complete the components
of~$\theta$ and~$\eta$ to a basis on the frame bundle cross~$\SO(3)$.

In particular, this formula yields that~$B_c$ 
is a Jacobi manifold for any choice of~$c=(c_1,c_2,c_3)$ 
and that its tableau has rank~$3$ when the $c_i$ are distinct, 
rank~$2$ when exactly two of the $c_i$ are equal, 
and rank~$0$ when all of the $c_i$ are equal.  

When all of the $c_i$ are equal, the tableau is trivial, and
so there is a regular flag (with characters $s_i=0$) by definition.

\subsubsection{Three distinct, constant eigenvalues}
When the $c_i$ are distinct, one sees that there is
a regular flag for the integral elements described by~$\pi_i=0$
with characters~$s_2=3$ and $s_i=0$ otherwise.  
In fact, a hyperplane in this integral element fails to be
the end of a regular flag if and only if it is described
by an equation of the form~$\xi = \xi_1\bar\eta_1 +\xi_2\bar\eta_2
 +\xi_3\bar\eta_3=0$ with $\xi_1\xi_2\xi_3 = 0$.
Consequently, Theorem~\ref{thm: CartanThmVar2} applies,
and it follows that, up to diffeomorphism, Riemannian $3$-manifolds
with distinct constant eigenvalues of the Ricci tensor 
depend on $3$ arbitrary functions of $2$ variables.

\subsubsection{Two distinct, constant eigenvalues}
However, when exactly two of the $c_i$ are equal,
there is no regular flag:  One easily checks that
the codimensions of the polar spaces of a generic flag 
for this tableau are $c_0 = c_1 = 0$, while $c_k = 2$ for~$k\ge 2$. 
However, the codimension of the space of integral elements 
is~$9>c_0+c_1+c_2+c_3+c_4+c_5$, 
as the reader can check.  Thus, when two of the $c_i$ are equal,
the $2$-dimensional Jacobi manifold~$B_c$ is not involutive.
  
This does not mean that there are not Riemannian metrics 
for which the Ricci tensor has two distinct, constant eigenvalues.
To check this, though, one must prolong the structure equations
and use Theorem~\ref{thm: CartanThmVar} 
instead of Theorem~\ref{thm: CartanThmVar2} as follows:

Suppose that~$\Ric(g) = {}^t\eta\circ R\circ\eta$ 
has two distinct constant eigenvalues, say~$c_1\not=c_2$ 
and $c_2$ (of multiplicity~$2$).  This means that there
is a circle bundle~$F^4$ over~$M^3$ consisting of the $g$-orthonormal
coframes such that $\Ric(g) = c_1\,{\eta_1}^2 + c_2\, \bigl({\eta_2}^2+{\eta_3}^2\bigr)$.  As the reader can check, this implies that
the structure equations on~$F$ can be written in the form
\begin{equation}\label{eq: R2consteigensforms}
\begin{aligned}
\d\eta_1    &= -2 a_1\,\eta_2\w\eta_3\\
\d\eta_2    &=    -\eta_{23}\w\eta_3 
                  -\bigl(a_2\,\eta_2 + (a_1{+}a_3)\,\eta_3\bigr)\w\eta_1\\
\d\eta_3    &= \phm\eta_{23}\w\eta_2 
                  +\bigl((a_1{-}a_3)\,\eta_2 + a_2\,\eta_3\bigr)\w\eta_1\\
\d\eta_{23} &= c_2\,\eta_2\w\eta_3\\
\end{aligned}
\end{equation}
where $a_1$, $a_2$, $a_3$ are functions 
satisfying~${a_1}^2-{a_2}^2-{a_3}^2 = \tfrac12c_1$ and the relations
\begin{equation}\label{eq: R2consteigensfunctions}
\begin{aligned}
\d a_1 &=&   && &&\hfill 2b_3\,\eta_2       &&+&&2b_4\,\eta_3\\
\d a_2 &=&     {}-2a_3\,\eta_{23}
               &&+&&(b_4{+}b_1)\,\eta_2&&+&&(b_3{+}b_2)\,\eta_3\\
\d a_3 &=& {}\phm 2a_2\,\eta_{23}
               &&-&&(b_3{-}b_2)\,\eta_2&&+&&(b_4{-}b_1)\,\eta_3\\
\end{aligned}
\end{equation}
for some functions~$b_1$, $b_2$, $b_3$, and $b_4$.  

Conversely, given an augmented coframing~$(a,\eta)$ 
satisfying the structure equations~\eqref{eq: R2consteigensforms} 
and~\eqref{eq: R2consteigensfunctions} 
and ${a_1}^2{-}{a_2}^2{-}{a_3}^2=\tfrac12c_1$, 
the form~$g ={\eta_1}^2+{\eta_2}^2+{\eta_3}^2$ 
defines a metric on the space of leaves of~$\eta_i=0$ that satisfies
$\Ric(g) = c_1\,{\eta_1}^2 + c_2\,\bigl({\eta_2}^2+{\eta_3}^2\bigr)$.

Now, because~$\d({a_1}^2-{a_2}^2-{a_3}^2) = \tfrac12 \d(c_1) = 0$,
the~$b_i$ must satisfy the relations
$$
\begin{aligned}
 a_2\,b_1&&{}+&&a_3\,b_2&&{}-&&(a_3{+}2a_1)\,b_3&&{}+&&         a_2\,b_4 &= 0,
\\
-a_3\,b_1&&{}+&&a_2\,b_2&&{}+&&         a_2\,b_3&&{}+&&(a_3{-}2a_1)\,b_4 &= 0,
\end{aligned}
$$
so that there are really only two `free derivatives' among the $b_i$,
as these two relations are independent 
except when~$(a_1,a_2,a_3)=(0,0,0)$ 
(and this can only happen if $c_1=0$; but when $c_1=0$, 
I will remove the locus where the $a_i$ all vanish 
from further consideration). 

The reader can check that there exist $1$-forms $\beta_i\equiv\d b_i\mod
\{\eta_1,\eta_2,\eta_3,\eta_{23}\}$ such that
$$
\begin{aligned}
 a_2\,\beta_1&&{}+&&a_3\,\beta_2&&{}-&&(a_3{+}2a_1)\,\beta_3&&
{}+&&         a_2\,\beta_4 &= 0,
\\
-a_3\,\beta_1&&{}+&&a_2\,\beta_2&&{}+&&         a_2\,\beta_3&&
{}+&&(a_3{-}2a_1)\,\beta_4 &= 0,
\end{aligned}
$$
and such that the relations
$$
\begin{aligned}
\d(\d a_1) &\equiv&   && &&\hfill 2\beta_3\w\eta_2 &&+&&2\beta_4\w\eta_3\\
\d(\d a_2) &\equiv&
               &&+&&(\beta_4{+}\beta_1)\w\eta_2&&+
            &&(\beta_3{+}\beta_2)\w\eta_3\\
\d(\d a_3) &\equiv&
               &&-&&(\beta_3{-}\beta_2)\w\eta_2&&+
              &&(\beta_4{-}\beta_1)\w\eta_3\\
\end{aligned}
$$
are identities modulo the above structure equations.

Meanwhile, the tableau of free derivatives is involutive, 
with $s_1 = 2$ and $s_i=0$ for $i>2$.  
Thus, Theorem~\ref{thm: CartanThmVar} applies, 
and one sees that the general such metric 
depends on $2$ functions of $1$ variable.

(For those who know about the characteristic variety, one can compute
that a covector is characteristic iff it is of the form $\xi = \xi_2\,\eta_2 + \xi_3\,\eta_3$ where $(\xi_2,\xi_3)$ satisfy
$$
(a_1{+}a_3)\,{\xi_2}^2 - 2a_2\,\xi_2\xi_3 + (a_1{-}a_3)\,{\xi_3}^2 = 0
$$
In particular, the characteristic variety consists 
of two complex conjugate points when $c_1>0$, 
a double point when~$c_1=0$, 
and two real distinct points when $c_1 < 0$.  
Consequently, the metrics with $c_1>0$ 
will be real-analytic in harmonic coordinates.)

\subsection{Torsion-free $H$-structures}
This last set of examples are applications to 
the geometry of $H$-structures on $n$-manifolds.

Let $\eum$ be a vector space over~$\bbR$ of dimension~$m$, 
and let $H\subset\GL(\eum)$ be a connected Lie subgroup of dimension~$r$
with Lie algebra $\euh\subset\eugl(\eum)= \eum\otimes\eum^*$.

One is interested in determining the generality, modulo
diffeomorphism, of the (local) $H$-structures that are torsion-free,
and, more generally, of torsion-free connections on $m$-manifolds 
with holonomy contained in (a conjugate of)~$H$.

\begin{remark}
When the first prolongation space of~$\euh$ vanishes, 
i.e., when
$$
\euh^{(1)}=(\euh\otimes\eum^*)\cap\bigl(\eum\otimes S^2(\eum^*)\bigr)
 = (0),
$$
these two questions are essentially the same, since, in this case,
an $H$-structure that is torsion-free 
has a unique compatible torsion-free connection,
while a torsion-free connection on~$M$ whose holonomy 
is conjugate to a subgroup~$K\subset H$ 
defines an $P/N$-parameter family of torsion-free $H$-structures, 
where $P\subset\GL(\eum)$ 
is the group of elements~$p\in\GL(\eum)$ such that $p^{-1}Kp\subset H$, 
while $N\subset H$ is the group of elements such that $p^{-1}Kp = K$.
\end{remark}

Now, the geometric objects being studied 
are the $H$-structures~$\pi:B\to M^m$ endowed 
with a torsion-free compatible connection.
Letting~$\eta:TB\to\eum$ be the canonical $\eum$-valued $1$-form on~$B$,
then the torsion-free compatible connection defines
an $\euh$-valued $1$-form~$\theta:TB\to\euh$ 
satisfying the \emph{first structure equation}
\begin{equation}\label{eq: Hstreq1}
\d\eta = -\theta\w\eta,
\end{equation}
and having the equivariance~$R_h^*(\theta)=\Ad(a^{-1})\bigl(\theta\bigr)$
for all~$h\in H$.  

One then has the \emph{second structure equation}
\begin{equation}\label{eq: Hstreq2}
\d\theta = -\theta\w\theta + \tfrac12\,R(\eta\w\eta)
\end{equation}
for a unique 
\emph{curvature function}~$R:B\to\euh\otimes\Lambda^2(\eum^*)$.

Conversely, any manifold~$B$ endowed with a coframing
$$
\omega=(\eta,\theta):TB\to \eum\oplus\euh = V
$$
satisfying the equations~\eqref{eq: Hstreq1} and~\eqref{eq: Hstreq2}
for some function~$R:B\to \euh\otimes \Lambda^2(\eum^*)$ is locally
diffeomorphic to the canonical coframing constructed above
from the data of an $H$-structure on a manifold~$M$ endowed
with a compatible, torsion-free connection.

Now, because ~$\d(\d\eta)=0$, the function~$R$ 
satisfies the \emph{first Bianchi identity},
$$
0 = \d(\d\eta) = 
-\d\theta\w\eta +\theta\w\d\eta = -(\d\theta+\theta\w\theta)\w\eta
= -\tfrac12\,R(\eta\w\eta)\w\eta = 0.
$$
I.e., $R$ takes values 
in the kernel~$K_0(\euh)\subset\euh\otimes\Lambda^2(\eum^*)$
of the natural map
$$
\euh\otimes \Lambda^2(\eum^*)\subset 
\eum\otimes\eum^*\otimes \Lambda^2(\eum^*)
\to \eum\otimes\Lambda^3(\eum^*).
$$
(This is the algebraic content of the first Bianchi identity.)

In particular, the combined structure equations~\eqref{eq: Hstreq1}
and~\eqref{eq: Hstreq2} define a system of equations 
for the coframing~$\omega = (\eta,\theta)$ 
taking values in~$V = \eum\oplus\euh$
for which the structure function is required to take values 
in an affine space~$A_{\euh}\subset V\otimes\Lambda^2(V^*)$ 
that is modeled 
on the linear subspace~$K_0(\euh)\subset\euh\otimes\Lambda^2(\eum^*)
\subset V\otimes\Lambda^2(V^*)$.

Differentiating~\eqref{eq: Hstreq2} yields, 
after some algebra, the \emph{second Bianchi identity}
$$
0 = \d(\d\theta) 
= \tfrac12\bigl(\d R+\rho'_0(\theta)R\bigr)(\eta\w\eta),
$$
where~$\rho_0:H\to\GL\bigl(K_0(\euh)\bigr)$ 
is the induced representation of~$H$ on~$K_0(\euh)$,
and~$\rho'_0:\euh\to\eugl\bigl(K_0(\euh)\bigr)$ 
is the induced map on Lie algebras.  This  means that
$$
\d R = - \rho'_0(\theta)R + R'(\eta),
$$
where~$R':B\to K_0(\euh)\otimes\eum^*$ takes values 
in the kernel~$K_1(\euh)\subset K_0(\euh)\otimes\eum^*$ 
of the natural linear mapping defined by skew-symmetrization
$$
K_0(\euh)\otimes\eum^*\subset \euh\otimes\Lambda^2(\eum^*)\otimes\eum^*
\to \euh\otimes \Lambda^3(\eum^*).
$$
This is the algebraic content of the second Bianchi identity.

In particular,~$A_{\euh}$ is a Jacobi manifold, 
and it is natural to ask when it is involutive,
which is a condition on the Lie algebra~$\euh\subset\eugl(\eum)$.
In fact, the test for involutivity is quite simple in this case:
One computes the characters~$s_i$ of~$K_0(\euh)$ 
considered as a tableau in~$\euh\otimes\Lambda^2(\eum^*)$.  
Then Cartan's bound implies that
$$
\dim K_1(\euh) \le s_1 + 2\,s_2 + \cdots + m\,s_m
$$
with equality if and only if $K_0(\euh)$, 
and, consequently, $A_{\euh}$ are involutive.  
Thus, this is a purely algebraic calculation.

\begin{example}[Riemannian metrics]
In the case that~$H=\SO(m)$, the structure equations take
the familiar form
$$
\d\eta_{i} = -\theta_{ij}\w\eta_j
$$
with $\theta_{ij}=-\theta_{ji}$ satisfying
$$
\d\theta_{ij} = -\theta_{ik}\w\theta_{kj} 
                    + \tfrac12R_{ijkl}\,\eta_k\w\eta_l\,,
$$
where the components of the Riemann curvature function~$R_{ijkl}$
satisfy the familiar relations~$R_{ijkl}=-R_{jikl}=-R_{ijlk}$
and $R_{ijkl}+R_{iklj}+R_{iljk}=0$.   
For the tableau~$K_0\bigl(\euso(m)\bigr)$, 
the character $s_p$ when~$1\le p\le m$ 
is the number of independent quantities~$R_{ijkp}$ 
subject to the above relations that have~$1\le k < p$, 
which one finds to be
$$
s_p = \tfrac12\,m(p-1)(m-p+1).
$$
(Of course, $s_p=0$ for~$m < p < \tfrac12m(m{+}1)$.)  As expected,
$$
s_1+\cdots+s_m = \tfrac1{12}m^2(m^2{-}1) = \dim K_0\bigl(\euso(m)\bigr)
$$
and one also finds
$$
s_1 + 2\,s_2 + \cdots + m\,s_m 
= \tfrac1{24}m^2(m^2{-}1)(m{+}2) = \dim K_1\bigl(\euso(m)\bigr),
$$
as this latter number is the number of independent~$R'_{ijklq}$ 
that show up in the formulae for the derivatives of the $R_{ijkl}$:
$$
\d R_{ijkl} = -R_{qjkl}\theta_{qi}-R_{iqkl}\theta_{qj}
              -R_{ijql}\theta_{qk}-R_{ijkq}\theta_{ql}
              + R'_{ijklq}\,\eta_q\,,
$$
which are subject to the classical second Bianchi identity
$R'_{ijklq}+R'_{ijqkl}+R'_{ijlqk}=0$.  

Thus, as expected, $A_{\euso(m)}$ is involutive, 
and the Riemannian metrics in dimension~$m$
(up to diffeomorphism) depend on~$s_m=\tfrac12m(m{-}1)$ 
functions of~$m$ variables.  The above characters 
then determine the number of independent covariant derivatives
of the curvature functions to any given order of differentiation.
\end{example}

\begin{example}[Ricci-flat K\"ahler surfaces]
When~$H=\SU(2)\subset\GL(4,\bbR)$, one is, in effect, 
considering Riemannian $4$-manifolds with holonomy contained in~$\SU(2)$.  In this case, one finds that $\dim K_0(\euh)=5$ 
and that the representation~$\rho_0$ of~$\SU(2)$ is irreducible.  
Indeed, one finds that the structure equations take the form
$$
\begin{pmatrix}\d\eta_0\\\d\eta_1\\\d\eta_2\\\d\eta_3\\ \end{pmatrix}
= -\begin{pmatrix}
    0 &\phm\theta_1 &\phm\theta_2 &\phm\theta_3\\
    -\theta_1&0 & -\theta_3 & \phm\theta_2\\
    -\theta_2& \phm\theta_3&0 & -\theta_1 \\
    -\theta_3&-\theta_2 & \phm\theta_1 & 0\\
    \end{pmatrix}
\w
\begin{pmatrix} \eta_0\\ \eta_1\\ \eta_2\\ \eta_3\\ \end{pmatrix}
$$
and 
$$
\begin{pmatrix}\d\theta_1\\\d\theta_2\\\d\theta_3\\ \end{pmatrix}
= 
-\begin{pmatrix}2\,\theta_2\w\theta_3\\
                2\,\theta_3\w\theta_1\\
                2\,\theta_1\w\theta_2\end{pmatrix}
+\begin{pmatrix}
    R_{11} &R_{12} &R_{13}\\
    R_{21} &R_{22} &R_{23}\\
    R_{31} &R_{32} &R_{33}\\
    \end{pmatrix}
\begin{pmatrix} \eta_0\w\eta_1-\eta_2\w\eta_3\\ 
                \eta_0\w\eta_2-\eta_3\w\eta_1\\ 
                \eta_0\w\eta_3-\eta_1\w\eta_2\\ \end{pmatrix},
$$
where~$R_{ij}=R_{ji}$ and $R_{11}+R_{22}+R_{33}=0$.  

It has already been shown that this defines a Jacobi manifold 
in $V\otimes\Lambda^2(V^*)$ where $V = \bbR^4\oplus\eusu(2)\simeq\bbR^7$,
and its involutivity follows by inspection,
since the characters are visibly $s_2=3$, $s_3=2$, and $s_k=0$ 
all other~$k$, and since the dimension of~$K_1(\euh)$ 
is easily computed to be~$12 = 2s_2 + 3s_3$.  

Thus, Theorem~\ref{thm: CartanThmVar2} applies 
and justifies Cartan's famous assertion that metrics in dimension~$4$
with holonomy~$\SU(2)$ depend on $s_3=2$ arbitrary functions 
of three variables up to diffeomorphism.%
\footnote{
``Les espaces de Riemann pr\'ec\'edents d\'ependent 
de deux fonctions arbitraires de trois arguments$...$'' 
(\cite{Cartan1925}, pp.~55--56).  As far as I know, 
Cartan never gave any justification for this assertion,
which is the earliest case I know of 
in which an irreducible holonomy group is discussed,
other than the case of symmetric spaces.  It seems highly
likely to me, though, that he was already, at that time (1926), 
aware of some version of Theorem~\ref{thm: CartanThmVar2}.
}
\end{example}

\begin{example}[Segre structures of dimension $2m$]
One can also apply these theorems to the study of `higher order'
$H$-structures, i.e., structures for which there is no canonical
connection until after a prolongation has been performed.

Consider the generality 
of torsion-free $\GL(2,\bbR){\cdot}\GL(m,\bbR)$-structures 
on~$\bbR^{2m}$.  In this discussion, I'm going to assume that $m>2$, 
since the case $m=2$ is equivalent to conformal structures of type $(2,2)$ 
on $\bbR^4$, which (as I'll point out below) 
turns out to have a different set of structure equations.

If~$F\to U\subset\bbR^{2m}$ is a torsion-free
$\GL(2,\bbR){\cdot}\GL(m,\bbR)$-structure on~$U\subset\bbR^{2m}$, 
then there is a prolongation of~$F$ to a second-order structure~$F^{(1)}$,
with structure group a semi-direct product of $\GL(2,\bbR){\cdot}\GL(m,\bbR)$
with~$\bbR^{2m}$, on which there exists a Cartan connection~$\theta$ 
with values in~$\SL(m{+}2,\bbR)$, say
$$
\theta = \begin{pmatrix}
\psi^i_j&\eta^i_\beta\\
\omega^\alpha_j&\phi^\alpha_\beta
\end{pmatrix},
$$
where the index ranges are understood 
to be~$1\le i,j,k\le 2$ and $1\le \alpha,\beta,\gamma\le m$, and  
the forms that are entries of~$\theta$
satisfy the single trace relation~$\psi^i_i+\phi^\alpha_\alpha = 0$ 
but are otherwise linearly independent.
These components are required to satisfy structure equations 
of the form%
\footnote{Here is where the assumption that $m>2$ is important.
The correct structure equations for $m=2$ have nontrivial 
curvature terms in the structure equations for $d\psi^i_j$, 
as the reader can easily check.  In fact, for $m=2$, the structure
equations as I have written them are the structure equations
for the so-called `half-flat' conformal structures of type $(2,2)$,
i.e., the ones for which the self-dual part of the Weyl curvature
vanishes.}
$$
\begin{aligned}
d\omega^\alpha_j 
&= -\phi^\alpha_\beta\w\omega^\beta_j -\omega^\alpha_i\w \psi^i_j\\
d\psi^i_j &= -\psi^i_k\w\psi^k_j - \eta^i_\beta\w\omega^\beta_j\\
d\phi^\alpha_\beta 
&= -\phi^\alpha_\gamma\w\phi^\gamma_\beta - \omega^\alpha_i\w\eta^i_\beta
   + F^\alpha_{\beta\gamma\delta}\,\omega^\gamma_1\w\omega^\delta_2\,\\
d\eta^i_\beta &= -\psi^i_j\w\eta^j_\beta-\eta^i_\alpha\w\phi^\alpha_\beta
   +G^i_{\beta\gamma\delta}\,\omega^\gamma_1\w\omega^\delta_2\,\\
\noalign{\vskip3pt}
dF^\alpha_{\beta\gamma\delta} 
&= -F^\epsilon_{\beta\gamma\delta}\,\phi^\alpha_\epsilon
+F^\alpha_{\epsilon\gamma\delta}\,\phi^\epsilon_\beta
+F^\alpha_{\beta\epsilon\delta}\,\phi^\epsilon_\gamma
+F^\alpha_{\beta\gamma\epsilon}\,\phi^\epsilon_\delta
+R^{\alpha i}_{\beta\gamma\delta\epsilon}\,\omega^\epsilon_i\\
dG^i_{\beta\gamma\delta} 
&= -G^j_{\beta\gamma\delta}\,\psi^i_j
+G^i_{\epsilon\gamma\delta}\,\phi^\epsilon_\beta
+G^i_{\beta\epsilon\delta}\,\phi^\epsilon_\gamma
+G^i_{\beta\gamma\epsilon}\,\phi^\epsilon_\delta
-F^\alpha_{\beta\gamma\delta}\,\eta^i_\alpha
+Q^{ij}_{\beta\gamma\delta\epsilon}\,\omega^\epsilon_j\,.\\
\end{aligned}
$$   
The functions~$F$, $G$, $R$, and $Q$ must satisfy the relations
$$
\begin{aligned}
F^\alpha_{\beta\gamma\delta} &= F^\alpha_{\gamma\beta\delta}
= F^\alpha_{\beta\delta\gamma},
\qquad\qquad F^\alpha_{\alpha\gamma\delta} = 0,\\
G^i_{\beta\gamma\delta} &= G^i_{\gamma\beta\delta}
= G^i_{\beta\delta\gamma}\\
\end{aligned}
$$
as well as the relations
$$
R^{\alpha i}_{\beta\gamma\delta\epsilon}
=P^{\alpha i}_{\beta\gamma\delta\epsilon}
+{1\over{m{+}3}}\left(
 \delta^\alpha_\beta\,G^i_{\gamma\delta\epsilon}
+\delta^\alpha_\gamma\,G^i_{\beta\delta\epsilon}
+\delta^\alpha_\delta\,G^i_{\beta\gamma\epsilon}
-(m{+}2)\delta^\alpha_\epsilon\,G^i_{\beta\gamma\delta}\right),
$$
where~$P^{\alpha i}_{\beta\gamma\delta\epsilon}$ 
is fully symmetric in its lower indices 
and satisfies~$P^{\alpha i}_{\alpha\beta\gamma\delta}=0$.  
Finally, $Q^{ij}_{\beta\gamma\delta\epsilon}$ 
must be fully symmetric in its lower indices. 

Note that, in the application of Theorem~\ref{thm: CartanThmVar2}, 
the $1$-forms play the role of the $\omega^i$, 
the independent coefficients in~$F$ and~$G$ 
play the role of coordinates on the appropriate Jacobi manifold~$A$, 
while the independent coefficients in~$P$ and~$Q$ 
play the role of coordinates on~$A^{(1)}$.

While the number~$n$ is actually~$(m{+}2)^2-1=m^2{+}4m{+}3$, 
it's also clear from the structure equations 
that only the~$\omega^\alpha_i$ are effectively involved 
in the computation of the characters (since it is only these terms 
that appear with non-constant coefficients in the structure equations). 
Thus (modulo what should be thought of as `Cauchy characteristics'), 
the `effective dimension' is~$n=2m$. 

As the reader can check, the formal $\d^2=0$ conditions
needed for Theorem~\ref{thm: CartanThmVar2} are satisfied.  
Using the symmetries of the coefficients, the dimensions
$$
\dim A = (m{+}2){{m{+}2}\choose 3} - {{m{+}1}\choose 2} 
 = {1\over 6}\, m (m{+}1) (m^2{+}4m{+}1)
$$ 
and
$$
\dim A^{(1)} = (2m{+}4){{m{+}3}\choose 4} - 2 {{m{+}2}\choose 3}
 = {1\over 12}\,m(m{+}1)(m{+}2)(m^2{+}5m{+}2)
$$
are easily computed.  

It remains to compute the characters, 
which turn out to be
$$s_{k}= (k{-}1)\bigl(m^2-(k{-}4)m - 2k + 3\bigr)$$ 
for~$1\le k\le m{+}1$ and~$s_{k}=0$ for~$k>m{+}1$.   
Thus, $A$ is an involutive Jacobi manifold.

In particular, up to diffeomorphism, the general such torsion-free structure 
depends on~$s_{m+1}=m(m{+}1)$ functions of~$m{+}1$ variables and
there exists such a structure taking any given desired curvature value.
\end{example}

\begin{remark}[Torsion-free $H$-structures]
For many other examples of this kind, examining the generality
up to diffeomorphism of local torsion-free $H$-structures 
for various groups~$H\subset\GL(m,\bbR)$,
the reader might consult~\cite{Bryant96} and~\cite{Bryant00}.
Essentially all questions about the existence and generality
of local torsion-free structures of this kind can be resolved 
by an application of Theorem~\ref{thm: CartanThmVar2}.
\end{remark}

Sometimes one wants to consider a proper 
submanifold of~$A_{\euh}$ in order to investigate $H$-structures
with some extra condition on the curvature that captures some
geometric property.  

\begin{example}[Einstein-Weyl structures]
Consider Cartan's analysis of the so-called Einstein-Weyl structures 
on $3$-manifolds. These structures are $\CO(3)$-structures 
on $3$-manifolds endowed with a compatible torsion-free connection 
whose curvature function takes values in a certain $4$-dimensional
submanifold~$W\subset A_{\euco(3)}$.  

Here are their structure equations as Cartan writes them 
(with a very slight change in notation):
$$
\begin{pmatrix}\d\eta_1\\\d\eta_2\\\d\eta_3\\ \end{pmatrix}
= -\begin{pmatrix}
     \phm\theta_0 & \phm\theta_3 &    -\theta_2 \\
        -\theta_3 & \phm\theta_0 & \phm\theta_1 \\
     \phm\theta_2 &    -\theta_1 & \phm\theta_0 \\
    \end{pmatrix}
\w
\begin{pmatrix} \eta_1\\ \eta_2\\ \eta_3\\ \end{pmatrix}
$$
and
$$
\begin{pmatrix}\d\theta_0\\\d\theta_1\\\d\theta_2\\\d\theta_3\\ 
  \end{pmatrix}
= 
 \begin{pmatrix}0\\
                \theta_2\w\theta_3\\
                \theta_3\w\theta_1\\
                \theta_1\w\theta_2\end{pmatrix}
+\begin{pmatrix}
    2H_1 & 2H_2 & 2H_3\\
     H_0 &  H_3 & -H_2\\
    -H_3 &  H_0 &  H_1\\
     H_2 & -H_1 &  H_0\\  
    \end{pmatrix}
\begin{pmatrix} \eta_2\w\eta_3\\ 
                \eta_3\w\eta_1\\ 
                \eta_1\w\eta_2\\ \end{pmatrix},
$$
where the functions~$H_0$, $H_1$, $H_2$, and $H_3$
are coordinates on~$W$.
This is a set of structure equations of the type 
to which Theorem~\ref{thm: CartanThmVar2} might apply, where the
affine subspace~$W\subset V\otimes\Lambda^2(V^*)$ has dimension~$4$
and where~$V = \bbR^3\oplus\bbR\oplus\euso(3)\simeq \bbR^7$.
It is easy to verify that $W$ is a Jacobi manifold and is involutive
with $s_2=4$ and all other $s_k=0$.  Thus, 
Theorem~\ref{thm: CartanThmVar2} applies, 
and one recovers Cartan's result that
the general Einstein-Weyl space depends on four arbitrary 
functions of two variables~\cite{Cartan1943}.
\end{example}

When $A_{\euh}$ is not involutive, one can ask whether 
its prolongation, which is got by adjoining the equation
\begin{equation}\label{eq: Hstreq3}
\d R = -\rho'_0(\theta)R + R'(\eta)
\end{equation}
to the pair~\eqref{eq: Hstreq1} and~\eqref{eq: Hstreq2}, 
is involutive, where $R'$ takes values 
in the subspace~$K_1(\euh)\subset K_0(\euh)\otimes\eum^*$ 
that is the kernel of the natural mapping
$$
K_0(\euh)\otimes\eum^*
\subset \euh\otimes\Lambda^2(\eum^*)\otimes\eum^*
\to \euh\otimes\Lambda^3(\eum^*).
$$
The combined system of equations~\eqref{eq: Hstreq1}, 
\eqref{eq: Hstreq2}, and~\eqref{eq: Hstreq3} 
is of the type that Theorem~\ref{thm: CartanThmVar} 
was intended to treat, with $R$ playing the role of the~$a^\alpha$ 
and $R'$ playing the role of the $b^\sigma$.  

It may be necessary to repeat this prolongation process several times 
in order to arrive at a system of structure equations 
to which Theorem~\ref{thm: CartanThmVar} can be applied.

\begin{example}[Bochner-K\"ahler metrics]
An interesting example is when~$\eum = \bbC^n$
and~$H = \Un(n)\subset \GL(\eum)$.  In this case, one finds
that $K_0(\euh)$ is decomposable as a $\Un(n)$-module 
into three irreducible summands,
$$
K_0(\euh) = S(\euh) \oplus \Ric_0(\euh)\oplus B(\euh),
$$
where $S(\euh)\simeq \bbR$ 
corresponds to the space of curvature tensors of K\"ahler manifolds 
with constant holomorphic sectional curvature, 
$\Ric_0(\euh)$ 
corresponds to the space of traceless Ricci curvatures of K\"ahler metrics, 
and $B(\euh)$, known as the space of \emph{Bochner curvatures}, 
corresponds to the space of curvature tensors 
of Ricci-flat K\"ahler manifolds. 
A K\"ahler metric is said to be \emph{Bochner-K\"ahler}
if the $B(\euh)$-component of its curvature tensor vanishes,
i.e., 
if its curvature tensor takes values in $\Ric_0(\euh)\oplus S(\euh)$. 

This defines a Jacobi manifold $A\subset K_0(\euh)$ that is
not involutive, but, after a succession of applications of the
prolongation process (in fact, three prolongations), 
one arrives at a set of structure equations that has no free
derivatives but satisfies the hypotheses of Theorem~\ref{thm: CTFT},
thus showing that germs of Bochner-K\"ahler metrics 
depend on a finite number of constants.  
For details, see~\cite{Bryant01}.
\end{example}

\bibliographystyle{hamsplain}

\providecommand{\bysame}{\leavevmode\hbox to3em{\hrulefill}\thinspace}

\end{document}